\documentclass{amsart}
\usepackage{amsmath,amsthm,latexsym,amssymb,mathrsfs}
\usepackage{mathabx}
\usepackage{mathtools}
\usepackage{dsfont}
\usepackage[english]{babel}
\usepackage[utf8]{inputenc}
\usepackage{color}
\usepackage{physics}
\usepackage{faktor}
\usepackage{overpic}
\usepackage{enumerate}
\usepackage[shortlabels]{enumitem}
\usepackage{pdflscape}
\usepackage{float}
\usepackage[alphabetic,nobysame]{amsrefs}
\usepackage{hyperref}
\usepackage[all,cmtip]{xy}
\usepackage{todonotes}
\usepackage{csquotes}
\usepackage{url}

\newtheoremstyle{mythm}% ⟨name⟩
{3pt}% ⟨Space above⟩
{3pt}% ⟨Space below⟩
{\itshape}% ⟨Body font⟩
{}% ⟨Indent amount⟩
{\bfseries}% ⟨Theorem head font⟩
{.}% ⟨Punctuation after theorem head⟩
{.5em}% ⟨Space after theorem head⟩
{\thmnote{#1 }#3}% ⟨Theorem head spec (can be left empty, meaning ‘normal’)⟩

\newtheorem{thm}{Theorem}[section]
\newtheorem*{thmx}{Theorem} %unlabeled theorem

\newtheorem{lem}[thm]{Lemma}
\newtheorem{cor}[thm]{Corollary}
\newtheorem{pro}[thm]{Proposition}
\newtheorem*{qn*}{Question}
\newtheorem*{lem*}{Lemma}

\newtheorem{thm*}{Theorem}
\newtheorem{pro*}[thm*]{Proposition}
\newtheorem{cor*}[thm*]{Corollary}

\newtheorem*{thm**}{Theorem}

\theoremstyle{mythm}

\theoremstyle{definition}
\newtheorem{dfn}[thm]{Definition}

\theoremstyle{remark}
\newtheorem{rmk}[thm]{Remark}
\newtheorem*{claim*}{Claim}
\newtheorem*{fact*}{Fact}
\newtheorem{claiminner}{Claim}  % manually labelled claim
\newenvironment{claim}[1]{%
    \renewcommand\theclaiminner{#1}%
    \claiminner
}{\endclaiminner}

\newcommand{\ep}{
    \epsilon
}
\newcommand{\mc}[1]{
    \mathcal{#1}
}
\newcommand{\mb}[1]{
    \mathbb{#1}
}

\title[Entropy, volume, and diameter of representations in ${\rm SO}(p,q+1)$]{Volume, entropy, and diameter in ${\rm SO}(p,q+1)$-higher Teichm\"uller spaces}

\author{Filippo Mazzoli}
\address{Max Planck Institute for Mathematics in the Sciences, Leipzig, Germany}
\email{filippomazzoli@me.com}

\author{Gabriele Viaggi}
\address{Department of Mathematics, Sapienza University of Rome, Rome, Italy}
\email{gabriele.viaggi@uniroma1.it}

\begin{document}
    
\begin{abstract}
We investigate properties of the pseudo-Riemannian volume, entropy, and diameter for convex cocompact representations $\rho : \Gamma \to \mathrm{SO}(p,q+1)$ of closed $p$-manifold groups. In particular: We provide a uniform lower bound of the product entropy times volume that depends only on the geometry of the abstract group $\Gamma$. We prove that the entropy is bounded from above by $p-1$ with equality if and only if $\rho$ is conjugate to a representation inside ${\rm S}({\rm O}(p,1)\times{\rm O}(q))$, which answers affirmatively to a question of Glorieux and Monclair. Lastly, we prove finiteness and compactness results for groups admitting convex cocompact representations with bounded diameter.
\end{abstract}

\maketitle

\section{Introduction}

A very rich interaction between topology, geometry, and dynamics occurs in the study of surfaces. A prototypical example is provided by hyperbolic structures and their moduli spaces, where the following classical properties hold:

\begin{enumerate}[(a)]
	\item \emph{Volume:} The Gau\ss-Bonnet formula relates the area of a hyperbolic surface and its Euler characteristic, a topological invariant.
	\item \emph{Entropy:} The topological entropy of the geodesic flow on a closed hyperbolic surface is always equal to $1$, independently of the hyperbolic structure.
	\item \emph{Compactness:} The moduli space is not compact and, by Mumford compactness criterion, the diameter defines a proper function on it.
\end{enumerate}

In this paper we intend to investigate how these phenomena generalize to the study of so-called \emph{higher Teichm\"uller spaces} associated with special orthogonal subgroups $\mathrm{SO}(p,q+1)$, as introduced by Danciger, Gu\'eritaud, and Kassel \cite{DGK}, Beyrer and Kassel \cite{BK}. Examples to keep in mind are:
\begin{itemize}
\item{The Teichm\"uller space of {\em maximal representations} of surface groups into ${\rm SO}(2,q+1)$ introduced by Burger, Iozzi, and Wienhard \cite{BIW}. See also Collier, Tholozan, and Toulisse \cite{CTT} and \cite{MV2} for a description of the pseudo-Riemannian geometry of these representations.}
\item{The Teichm\"uller space of continuous deformations of the composition of the inclusion of a uniform lattice $\Gamma<{\rm SO}(p,1)$ in ${\rm SO}(p,1)$ with the upper block embedding ${\rm SO}(p,1)\to{\rm SO}(p,q+1)$. See Barbot \cite{B15} or \cite{BK} for the case $q=1$.}
\end{itemize}

These spaces, in particular, parametrize objects of the form 
\[
\mathsf{N}_\rho : = \Omega_{\rho}/\rho(\Gamma) ,
\]
where: $\rho : \Gamma \to \mathrm{SO}(p,q+1)$ is a faithful representation with discrete image, $\Gamma$ is a fixed hyperbolic group of cohomological dimension $p$, and $\Omega$ is a maximal $\Gamma$-invariant open convex subset of the {\em pseudo-Riemannian hyperbolic space} $\mathbb{H}^{p,q}$. Any such representation is called $\mathbb{H}^{p,q}$-{\em convex cocompact}, in the sense of \cite{DGK}. One can think of $\mathsf{N}_\rho$ as generalizing the role of a hyperbolic surface from above, and of the space of conjugacy classes
\[
\mathcal{CC}(\Gamma) : =  \{ [\rho] \mid \text{$\rho : \Gamma \to \mathrm{SO}(p,q+1)$ is convex cocompact}\}/{\rm Out}(\Gamma) ,
\]
where ${\rm Out}(\Gamma)$ denotes the group of outer automorphisms of $\Gamma$, as generalizing the role of the Riemann moduli space.

Any such $\mathsf{N}_\rho$ has natural pseudo-Riemannian {\em entropy} $\delta_\rho$, {\em volume} ${\rm vol}_\rho$, and {\em diameter} ${\rm diam}_\rho$. We establish three properties about these invariants: 
\begin{enumerate}[(a)]
	\item{{\em Volume, entropy, and topology:} We give uniform lower bounds on the product $\delta_\rho^p \cdot \mathrm{vol}_\rho$ only depending on the geometry of the group $\Gamma$, see Theorem~\ref{thm:main2}.}
	\item{{\em Entropy rigidity:} The entropy is bounded from above by $p-1$ with equality if and only if $\rho(\Gamma)$ is conjugate into ${\rm S}({\rm O}(p,1)\times{\rm O}(q))$, see Theorem~\ref{thm:main1}. This answers a question of Glorieux and Monclair \cite{GM18}*{Question~1.6}.}
	\item{{\em Diameter and complexity:} We prove finiteness and compactness results for groups $\Gamma$ admitting convex cocompact representations with bounded diameter ${\rm diam}_\rho\le D$, see Theorems~\ref{thm:main4}, \ref{thm:main5}, and \ref{thm:main6}. }
\end{enumerate}

In the rest of the introduction, we state and comment each of these results, but first we review some basic facts about $\mb{H}^{p,q}$ and ${\rm SO}(p,q+1)$-{\em higher Teichmüller spaces}.

\subsection*{\texorpdfstring{${\rm SO}(p,q+1)$}{SO(p,q+1)}-Convex cocompactness}
Let us briefly recall the setup: ${\rm SO}(p,q+1)$ is the group of orientation-preserving linear isometries of $\mb{R}^{p,q+1}$, which denotes the vector space $\mb{R}^{p+q+1}$ endowed with the quadratic form of signature $(p,q+1)$
\[
\langle\bullet,\bullet\rangle_{p,q+1}=x_1^2+\cdots+x_p^2-y_1^2-\cdots-y_{q+1}^2.
\]

The group ${\rm SO}(p,q+1)$ acts on the pseudo-Riemannian homogeneous space
\[
\mb{H}^{p,q}=\{[x]\in\mb{P}(\mb{R}^{p,q+1})|\;\langle x,x\rangle_{p,q+1}<0\} 
\]
by isometries with respect to the natural pseudo-Riemannian metric of signature $(p,q)$ induced by $\langle\bullet,\bullet\rangle_{p.q+1}$, and on its \emph{ideal boundary}
\[
\partial \mathbb{H}^{p,q} = \{[z] \in \mathbb{P}(\mathbb{R}^{p,q+1}) \mid \langle z, z \rangle_{p,q+1} = 0 \} 
\]
by preserving its associated conformal structure of signature $(p-1,q)$.

Geodesics in $\mathbb{H}^{p,q}$ are all of the form $\mathbb{P}(V) \cap \mathbb{H}^{p,q}$, where $V$ is some $2$-dimensional subspace of $\mathbb{R}^{p+q+1}$ whose associated projective line $L : =\mathbb{P}(V)$ intersects $\mathbb{H}^{p,q}$. Any geodesic is of one of the following kinds, depending on its induced metric: $L \cap \mathbb{H}^{p,q}$ is either \emph{spacelike}, \emph{timelike}, or \emph{lightlike} depending on whether the restriction of the pseudo-Riemannian metric $g_{\mathbb{H}^{p,q}}$ is positive, negative, or degenerate, respectively. Moreover, we say that a subset $C \subset \mathbb{H}^{p,q}$ is \emph{convex} if it is contained inside an affine chart of $\mathbb{P}(\mathbb{R}^{p+q+1})$ in which it is convex as a subset of the affine space. 

The action of ${\rm SO}(p,q+1)$ preserves also a {\em pseudo-distance} defined as follows: Consider a pair of distinct points $[x],[y]\in\mb{H}^{p,q}$ and let $V:={\rm Span}\{x,y\}$ be the 2-dimensional subspace they generate. We say that $[x], [y] \in \mathbb{H}^{p,q}$ are \emph{space-related} if $V$ has signature $(1,1)$ (or, equivalently, if the line $\mathbb{P}(V) \cap \mathbb{H}^{p,q}$ is spacelike). Then we set
\[
d_{\mb{H}^{p,q}}([x],[y]):=
\left\{
\begin{array}{ll}
	{\rm arccosh}\left(-\frac{\langle x,y\rangle}{\sqrt{\langle x,x\rangle\langle y,y\rangle}}\right) &\text{if $[x], [y]$ are space-related},\\
	0 &\text{otherwise},\\
\end{array}
\right.
\]
where the representatives $x, y \in \mathbb{R}^{p,q+1}$ are chosen so that $\langle x, y \rangle < 0$. We remark that $d_{\mathbb{H}^{p,q}}([x],[y])$ coincides with the length of the unique spacelike geodesic segment in $\mathbb{H}^{p,q}$ joining $[x], [y] \in \mathbb{H}^{p,q}$.

Following Danciger, Gu\'eritaud, and Kassel \cite{DGK}, a representation $\rho : \Gamma \to \mathrm{SO}(p,q+1)$ is said to be $\mathbb{H}^{p,q}$-{\em convex cocompact} (or simply \emph{convex cocompact} throughout our exposition) if it has finite kernel and if there exists a $\rho$-invariant convex subset $C$ of $\mathbb{H}^{p,q}$ that satisfies:
\begin{itemize}
	\item The group $\Gamma$ acts cocompactly on $\overline{C}\cap \mathbb{H}^{p,q}$, where $\overline{C}$ denotes the closure of $C$ inside $\mathbb{P}(\mathbb{R}^{p,q+1})$.
	\item The ideal boundary $\partial C : = \overline{C} \cap \partial \mathbb{H}^{p,q}$ does not contain any open lightlike segments.
\end{itemize} 

Furthermore, there exists a unique $\rho$-invariant \emph{maximal open, convex domain of discontinuity} $\Omega_\rho \subset \mathbb{H}^{p,q}$ that contains $C$. More precisely, $\Omega_\rho$ is the connected component of $\mathbb{H}^{p,q} - \bigcup_{z} \mathbb{P}(\langle z \rangle^\perp)$ containing $C$, as $z$ varies in $\partial C$.

By \cite{DGK}*{Theorem~1.24}, any group $\Gamma$ that admits a convex cocompact representation is word-hyperbolic, and the natural inclusion $\Gamma \subset \mathrm{SO}(p,q+1)$ is $P_1$-Anosov. Recall also from Danciger, Gu{\'e}ritaud, and Kassel \cite{DGK} that convex cocompactness is an open condition inside $\mathrm{Hom}(\Gamma,\mathrm{SO}(p,q+1))$.

\subsection*{\texorpdfstring{${\rm SO}(p,q+1)$}{SO(p,q+1)}-Higher Teichmüller spaces}

The recent works of Seppi, Smith, and Toulisse \cite{SST23} and of Beyrer and Kassel \cite{BK} have shed new light on the geometry of convex cocompact representations in $\mathrm{SO}(p,q+1)$. By \cite{SST23}, for any group $\Gamma$ of cohomological dimension $p$ that admits a convex cocompact representation $\rho : \Gamma \to \mathrm{SO}(p,q+1)$, there exists a unique $\rho$-invariant spacelike, complete $p$-submanifold $M \subset \mathbb{H}^{p,q}$ that is \emph{maximal}, namely whose trace of the second fundamental form is equal to zero. This extends previous works of Andersson, Barbot, B\'eguin, and Zeghib \cite{ABBZ}, Bonsante and Schlenker \cite{BS10}, Collier, Tholozan, and Toulisse \cite{CTT}, and Labourie, Toulisse, and Wolf \cite{LTW} and includes them in a unified framework. Furthermore, under the same hypotheses, any continuous deformation of $\rho:\Gamma\to{\rm SO}(p,q+1)$ within ${\rm Hom}(\Gamma,{\rm SO}(p,q+1))$ is again convex cocompact by \cite{BK}. It follows that the space of conjugacy classes of convex cocompact representations of $\Gamma$ is a union of connected components and, hence, a higher-rank Teichm\"uller space.

The main aim of the present paper is to exploit some of the key features of maximal spacelike submanifolds of $\mathbb{H}^{p,q}$ to investigate the global structure of $\mathrm{SO}(p,q+1)$-higher Teichm\"uller spaces. In particular, our work leverages on the following properties.

\begin{thmx}[{Ishihara~\cite{I}*{Proposition~2.1, Theorem~1.2}}]
	Let $M$ be a complete, spacelike, maximal $p$-submanifold of $\mathbb{H}^{p,q}$. Then its Ricci curvature and second fundamental form $\mathbb{I}$ satisfy
	\begin{equation*}
		\mathrm{Ric}_M + (p - 1) \, g_M \geq 0 , \qquad \norm{\mathbb{I}}_\infty \leq pq , \tag{I} \label{eq:ricci bound}
	\end{equation*}
	where $g_M$ denotes the induced metric of $M$.
\end{thmx}

\begin{thm*}\label{thm:distance comparison}
	Let $M$ be a complete, spacelike, maximal $p$-submanifold of $\mathbb{H}^{p,q}$. Then the distance $d_M$ associated with the induced metric of $M$ and the restriction to $M$ of the pseudo-metric $d_{\mathbb{H}^{p,q}}$ satisfy 
	\begin{equation*}
		d_M(x,y) \leq d_{\mathbb{H}^{p,q}}(x,y) \leq \sqrt{p} \, d_M(x,y) , \qquad x, y \in M .  \tag{II} \label{eq:distance comparison}
	\end{equation*}
\end{thm*}

Theorem~\ref{thm:distance comparison} is a new contribution that generalizes Labourie and Toulisse \cite{LT22}*{Theorem~6.5} and it will be deduced from Theorem~\ref{thm:main3} below. (Notice that Theorem~\ref{thm:distance comparison} does not require $M$ to be invariant by the action of some convex cocompact representation.) 

More precisely, we apply properties \eqref{eq:ricci bound} and \eqref{eq:distance comparison} of maximal submanifolds to study volume, entropy, and diameter of convex cocompact representations. Given any such $\rho : \Gamma \to \mathrm{SO}(p,q+1)$ with associated pseudo-hyperbolic manifold $\mathsf{N}_\rho : = \Omega_\rho / \rho(\Gamma)$, we consider:

\begin{enumerate}[(a)]
	\item The \emph{volume} of $\mathsf{N}_\rho$
	\[
	\mathrm{vol}_\rho : = \mathrm{vol}\left(\mathsf{N}_\rho\right) = \int_{\mathsf{N}_\rho} \abs{\det(g_{\mathsf{N}_\rho}){}_{i,j}}^{1/2} \wedge_i \dd x_i ,
	\]
	where the integrand equals the volume form induced by the pseudo-Riemannian metric of $N_{\rho}$ (in local coordinates $(x_i)_i$).
	\item The \emph{entropy} (or \emph{pseudo-Riemannian critical exponent}) of $\rho$
	\[
	\delta_\rho : = \limsup_{R \to \infty} \frac{1}{R} \, \log\abs{\{\gamma \in \Gamma \mid d_{\mathbb{H}^{p,q}}(o, \rho(\gamma) o) \leq R \}} ,
	\]
	for some (any) choice of a basepoint $o \in \Omega_\rho$. 
	\item The \emph{diameter} of $\mathsf{N}_\rho$
	\[
	\mathrm{diam}_\rho : = \sup \left\{ \left. \inf_{\gamma \in \Gamma} d_{\mathbb{H}^{p,q}}(x, \rho(\gamma)y) \, \right| \, \text{$x, y \in \Omega_\rho$ are space-related} \right\} .
	\]
\end{enumerate}

Let us provide a brief motivation for our interest in the study of these quantities. The volume $\mathrm{vol}_\rho$ has been extensively studied when $q = 0$ and $\rho$ is cocompact (in which case $\mathsf{N}_\rho$ is a closed hyperbolic $p$-manifold and $\mathrm{vol}_\rho$ is its Riemannian volume), and when $q = 1$. For the latter case, recall that $\rho : \Gamma \to \mathrm{SO}(p,2)$ is $\mathbb{H}^{p,1}$-convex cocompact if and only if $\mathsf{N}_\rho$ is a \emph{maximal globally hyperbolic, Cauchy-compact} (or, more concisely, \emph{MGHC}), \emph{anti-de Sitter spacetime} of dimension $p + 1$, by the work of Barbot and M\'erigot \cite{BM} when $\Gamma$ is the fundamental group of a closed hyperbolic $p$-manifold, and by \cite{DGK} for a general word hyperbolic group. In this context, $\mathrm{vol}_\rho$ is simply the Lorentzian volume of $\mathsf{N}_\rho$. For $p = 2$ and $q = 1$, we refer to the work of Bonsante, Seppi, and Tamburelli \cite{BST}, who observed multiple interesting relations between $\mathrm{vol}_\rho$ and classical hyperbolic geometry in dimension $2$.

The critical exponent $\delta_\rho$ was first introduced and studied by Glorieux and Monclair \cite{GM18}. In their work, the authors showed that $\delta_\rho$ coincides both with the dimension of suitable (natural adaptation of the notion of) Patterson-Sullivan densities, and with an appropriate pseudo-Riemannian analog of Hausdorff dimension of the limit set of $\rho$. See also Collier, Tholozan, Toulisse \cite{CTT} and Theorem~\ref{thm:main1} below for further comments and applications. 

We introduce here the diameter $\mathrm{diam}_\rho$ as a natural geometric invariant associated with the pseudo-Riemannian manifold $\mathsf{N}_\rho$. On the one hand, its definition does not directly involve the Riemannian geometry of the maximal submanifold $\mathsf{M}_\rho = M/\rho(\Gamma)$, but rather the dynamics of the action by $\rho$ on the homogeneous space $\mathbb{H}^{p,q}$, in analogy with the critical exponent $\delta_\rho$. On the other hand, both $\mathrm{vol}_\rho$ and $\mathrm{diam}_\rho$ are related to classical Riemannian invariants of the maximal submanifold $\mathsf{M}_\rho$, as described by the following statement.

\begin{pro*}
\label{pro:finite vol diam}
	Let $\rho : \Gamma \to \mathrm{SO}(p,q+1)$ be a convex cocompact representation. Then the following properties hold:
	\begin{enumerate}
		\item There exists a dimensional constant $\mu = \mu(p,q) > 0$ such that
		\[
		0 < \mu \, \mathrm{vol}(\mathsf{M}_\rho) \leq \mathrm{vol}_\rho < \infty .
		\] 
		\item The volume entropy $h(\mathsf{M}_\rho)$ of $\mathsf{M}_\rho$ satisfies
		\[
		h(\mathsf{M}_\rho) \leq \delta_\rho \leq \frac{1}{\sqrt{p}} \, h(\mathsf{M}_\rho) .
		\]
	\end{enumerate}
\end{pro*}

(See Corollary~\ref{cor:finite invariants} and Proposition~\ref{pro:volume} for a proof of item (1). Item (2) is a direct consequence of Theorem~\ref{thm:distance comparison}.) In the remainder of the introduction, we describe and comment the statements of the main theorems.

\subsection*{Volume and entropy}
Our first main result relates volume, entropy, and topology of convex cocompact representations.

\begin{thm*}
\label{thm:main2}
There exists a constant $c=c(p,q)>0$ such that the following holds. Let $\Gamma$ be the fundamental group of a closed orientable $p$-manifold $\mathsf{M}$. Then for every convex cocompact representation $\rho : \Gamma \to \mathrm{SO}(p,q+1)$
\[
\delta_\rho^p\cdot{\rm vol}_\rho\ge c \cdot \norm{\mathsf{M}},
\]
where $\norm{\mathsf{M}}$ denotes the {\rm simplicial volume} of $\mathsf{M}$. 
\end{thm*}

Comments:
\begin{itemize}
\item{In fact, it follows from our proof that a sharper version of Theorem~\ref{thm:main2} holds by replacing $\mathrm{vol}_\rho$ with $\mathrm{vol}(\mathsf{M}_\rho)$, the volume of the maximal spacelike $p$-submanifold of $\mathsf{N}_\rho$.}
\item{The {\em simplicial volume} is a topological invariant of $\mathsf{M}$ (in fact it only depends on the fundamental group $\Gamma$) introduced by Gromov in \cite{Gr}. Heuristically speaking, it measures the minimal number of singular simplices needed to represent the fundamental class $[\mathsf{M}]\in H_p(\mathsf{M},\mb{R})$.}
\item{By \cite{DGK}, convex cocompactness of $\rho$ implies that $\Gamma$ is word-hyperbolic and, hence, the simplicial volume $\norm{\mathsf{M}}$ is positive by results of Mineyev \cite{M}. In the case of a hyperbolic manifold $\mathsf{M}$, we have $\norm{\mathsf{M}}={\rm vol}(\mathsf{M})/v_p$ where $v_p$ is the volume of a regular ideal $p$-simplex in $\mb{H}^p$.}
\item{The inequality can be seen as a pseudo-Riemannian analog of a very general volume-entropy bound due to Gromov \cite{Gr}, who showed that there exists a constant $\kappa=\kappa(p)>0$ such that, for every closed, orientable, Riemannian $p$-manifold $\mathsf{M}$, we have 
\[
h(\mathsf{M})^p\cdot{\rm vol}(\mathsf{M})\ge\kappa\cdot \norm{\mathsf{M}} ,
\]
where $h(\mathsf{M})$ is the {\em volume entropy} of the universal cover of $\mathsf{M}$ and $\norm{\mathsf{M}}$ is the simplicial volume of $\mathsf{M}$.}
\end{itemize}

\subsection*{Entropy rigidity}

Next, we prove the following rigidity statement for the pseudo-Riemannian entropy:

\begin{thm*}
	\label{thm:main1}
	Let $\Gamma$ be the fundamental group of a closed orientable $p$-manifold $\mathsf{M}$. Then for every convex cocompact representation $\rho : \Gamma \to \mathrm{SO}(p,q+1)$
	\[
	\delta_\rho\le p-1 ,
	\]
	with equality if and only if $\rho$ preserves a totally geodesic copy of $\mb{H}^p$ in $\mb{H}^{p,q}$.
\end{thm*}

Comments:
\begin{itemize}
	\item{The inequality in Theorem~\ref{thm:main1} was originally proved by Glorieux and Monclair \cite{GM18} and by now there are several generalizations and refinements of it, see e.g. \cites{C20,PSW23,GMT}. Our main contribution concerns the rigidity in case of equality, which answers positively to \cite{GM18}*{Question~1.6}.}
	\item{The existence of maximal submanifolds from \cite{SST23}, together with properties \eqref{eq:ricci bound} and \eqref{eq:distance comparison}, allows to translate the entropy rigidity problem into a purely Riemannian framework, where one can apply the work of Ledrappier and Wang \cite{LW09}.}
	\item{If $\delta_\rho=p-1$, then $\rho$ preserves a totally geodesic $\mb{H}^p\subset\mb{H}^{p,q}$ and, hence, it is conjugate into ${\rm S}({\rm O}(p,1)\times{\rm O}(q))$.}
	\item{It is known that there exist closed orientable $p$-manifolds, whose fundamental groups $\Gamma$ admit convex cocompact representations inside $\mathrm{SO}(p,q+1)$, that are not virtually isomorphic to uniform lattices in ${\rm SO}(p,1)$, by the works of Monclair, Schlenker, and Tholozan \cite{MST} and Lee and Marquis \cite{LM}. For them we necessarily have $\delta_\rho<p-1$.}
	\item{When $p = 2$, this statement is due to Collier, Tholozan, and Toulisse \cite{CTT}, see also \cite{MV2} for a different proof.}
	\item{Theorem~\ref{thm:main1} is a pseudo-Riemannian analog of various entropy rigidity results, such as the works of Besson, Courtois, and Gallot \cite{BCG}, Hamenst\"{a}dt \cite{H}, Ledrappier and Wang \cite{LW09}, and Barthelm\'{e}, Marquis, and Zimmer \cite{BMZ} in Riemannian, hyperbolic, and convex projective geometries, respectively.}
\end{itemize}

\subsection*{Diameter and complexity}
Lastly, we relate the diameter of $\mathsf{N}_\rho = \Omega_\rho/\rho(\Gamma)$ to the complexity of the group $\Gamma$.

\begin{thm*}
\label{thm:main4}
For every $p \geq 2$ and $D > 0$, there exist at most $n$ homeomorphism classes of closed, aspherical $p$-manifolds whose fundamental groups $\Gamma$ admit \emph{faithful} convex cocompact representations $\rho : \Gamma \to \mathrm{SO}(p,q+1)$ with pseudo-Riemannian diameter ${\rm diam}_\rho\le D$, for some $n = n(D,p) \in \mathbb{N}$ independent of $q \geq 1$.
\end{thm*}

Comments:
\begin{itemize}
\item{In fact, Theorem~\ref{thm:main4} applies also if we replace the function $\mathrm{diam}_\rho$ with $\mathrm{diam}(\mathsf{M}_\rho)$, the diameter of the maximal spacelike $p$-submanifold of $\mathsf{N}_\rho$.}
\item{This a pseudo-hyperbolic analog of various finiteness results in hyperbolic geometry, see for example Thurston \cite{Th}*{\S~5.11}, Canary, Epstein, and Green \cite{CEG}*{\S~I.3}, Benedetti and Petronio \cite{BP92}*{\S~E}.}
\item{By work of Thurston, there exist infinitely many closed hyperbolic 3-manifolds of volume bounded by $V$ (see \cite{Th}*{Chapter~4}). On the other hand, if $p\ge 4$, Wang \cite{W} proved that the number of closed hyperbolic $p$-manifolds with volume bounded from above by $V$ is finite. It seems natural to ask whether this holds true also for the number of homeomorphism classes of closed, aspherical $p$-manifolds whose fundamental groups admit faithful convex cocompact representations in ${\rm SO}(p,q+1)$ with $p\ge 4$ and pseudo-Riemannian volume ${\rm vol}_\rho\le V$.}
\end{itemize}

As opposed to the Riemannian setting, there are two central difficulties that one has to deal with when working in $\mb{H}^{p,q}$:
\begin{enumerate}
\item{The function $d_{\mb{H}^{p,q}}(\bullet,\bullet)$ is not a distance. For instance, metric balls with respect to $d_{\mb{H}^{p,q}}$ are not compact.}
\item{Point stabilizers for the action ${\rm SO}(p,q+1)\curvearrowright\mb{H}^{p,q}$ are not compact.}
\end{enumerate}

This absence of compactness explains what goes wrong if one tries to approach the problem with standard Riemannian tools. In order to circumvent the issue, we need to find suitable replacement for the two properties.

\begin{thm*}
\label{thm:main5}
Let $\Gamma$ be the fundamental group of a closed, aspherical $p$-manifold. Denote by $\mathcal{CC}(\Gamma)$ the space of convex cocompact representations of $\Gamma$ in ${\rm SO}(p,q+1)$ up to conjugation and up to outer automorphisms of $\Gamma$. Then the diameter function
\[
\begin{matrix}
	\mathcal{CC}(\Gamma) & \longrightarrow & (0, \infty) \\
	[\rho] & \longmapsto & {\rm diam}_{\rho} ,
\end{matrix}
\]
is proper.
\end{thm*}

In fact, the proof of Theorem~\ref{thm:main5} shows that the same holds if we replace the role of the function $\mathrm{diam}_\rho$ with the diameter of the maximal spacelike $p$-submanifold $\mathsf{M}_\rho$. Note that {\em bending} examples show that the moduli space $\mc{CC}(\Gamma)$ is {\em non-compact} for every $\Gamma$ which is the fundamental group of a closed hyperbolic $p$-manifold containing a totally geodesic hypersurface. We also remark that, by standard arithmetic constructions, it is possible to construct infinitely many different commensurability classes of groups $\Gamma$ with these properties.  

The statement is a pseudo-Riemannian compactness criterion which is analogue to Mumford's Compactness Theorem \cite{Mum} for hyperbolic surfaces. In its proof we use that the set of convex cocompact representations of a given $\Gamma$ is closed in the space of representations, as shown by Beyrer and Kassel \cite{BK}.

Recall that, when $p=2$, convex cocompact representations of a fixed surface group $\Gamma$ into ${\rm SO}(2,q+1)$ coincide with the class of so-called {\em maximal representations} (as defined in \cite{BIW}, see also \cite{CTT}) and, by the Dehn-Nielsen-Baer Theorem (see e.g. \cite{FM}*{Theorem~8.1}), ${\rm Out}(\Gamma)$ is the (extended) {\em mapping class group} of the surface. Thanks to the better control in such low dimensional setting, we can replace the diameter with \emph{systole} and \emph{volume} functions, where the systole of a convex cocompact representation $\rho$ is defined by 
\[
{\rm sys}_\rho:=\min_{\gamma\in\Gamma-\{1\}}\ell_\rho(\gamma).
\]

We have the following.

\begin{thm*}
\label{thm:main6}
Let $S$ be a closed orientable surface of genus at least $2$ and let $\pi = \pi_1(S)$ denote its fundamental group. For every $\ep,V>0$ and $q\in\mb{N}$, the mapping class group ${\rm Out}(\pi)$ acts cocompactly on
\[
\mathcal{M}(q,\ep,V)=\left\{\rho\in{\rm Hom}(\pi,{\rm SO}(2,q+1))\left|
\begin{array}{c}
\rho\,\,{\rm maximal},\\ 
{\rm vol}_{\rho}\le V, {\rm sys}_{\rho}\ge\ep\\
\end{array}
\right.\right\} / \mathrm{SO}(2,q+1) .
\]
\end{thm*}

We conclude the introduction with a few words on the ingredients of the proofs.

\subsection*{Spatial distances on maximal submanifolds}

As we mentioned before, the key input for Theorems~\ref{thm:main2}, \ref{thm:main1}, and \ref{thm:main4} are properties \eqref{eq:ricci bound} and \eqref{eq:distance comparison} of ($\rho$-invariant) spacelike maximal $p$-submanifolds of $\mathbb{H}^{p,q}$, whose existence and uniqueness has been established by the recent work of Seppi, Smith, and Toulisse \cite{SST23}. In turn, property \eqref{eq:distance comparison} follows from the next crucial statement:

\begin{thm*}
\label{thm:main3}
Let $M$ be a complete, spacelike, maximal $p$-submanifold of $\mathbb{H}^{p,q}$, and denote by $\beta : M \to \mb{R}$ the restriction to $M$ of either the pseudo-Riemaniann distance from some $x\in M$ or the pseudo-Riemannian Busemann function associated with some $\theta\in\partial M$. Then:
\[
1 \leq \norm{\nabla \beta}_{M,\infty}\le\sqrt{p} ,
\]
where $\norm{\bullet}_{M,\infty}$ denotes the $L^\infty$-norm on $M$. Moreover, 
\begin{enumerate}
	\item $\norm{\nabla \beta}_{M,\infty} = 1$ if and only if $M$ is a totally geodesic copy of $\mathbb{H}^p$.
	\item If there exists $y \in M$ such that $\norm{\nabla \beta}_{y} = \sqrt{p}$, then $\nabla^2 \beta|_y \equiv 0$.
\end{enumerate}
\end{thm*}

Comments:
\begin{itemize}
\item{Compared with the Riemanniann setting, where distance functions have gradient of norm $\le 1$ when restricted to submanifolds, for the {\em pseudo-Riemannian Busemann function} there is no such bound (see also Proposition~\ref{pro:busemann}). What we recover is a uniform bound for the restriction to the {\em maximal} $p$-submanifold.}
\item{This result generalizes the work of Labourie and Toulisse \cite{LT22}*{Proposition~5.8} for $p=2$ to any dimension $p\ge 2$.}
\item{Concerning the equality case, one should remark that Labourie and Toulisse \cite{LT22} show that, when $p = 2$ and $\norm{\nabla\beta}_y = \sqrt{2}$, then the scalar curvature at $y$ equals zero. From this, generalizing Ishihara \cite{I} and Cheng \cite{C94}, the authors prove that the scalar curvature must vanish identically and, therefore, the surface must be a Barbot surface. One may ask to what extent this phenomenon generalizes to higher dimension. In this direction, upcoming work of Moriani and Trebeschi \cite{MT} establishes that, when the scalar curvature of $M$ vanishes at some point, then $M$ has to be a maximal Barbot $p$-submanifold. In particular, when $q=1$, their work implies that Case (2) of Theorem~\ref{thm:main3} arises only when $M$ is a Barbot submanifold of $\mb{H}^{p,1}$ which does not occur for $\mb{H}^{p,1}$-convex cocompact representations.}
\end{itemize}

\subsection*{Acknowledgements}
We thank Andrea Seppi and Jérémy Toulisse for their support and encouragement to write this paper. We are also grateful to Francesco Bonsante, Sara Maloni, Alex Moriani, Beatrice Pozzetti, Enrico Trebeschi, and Anna Wienhard for their feedback and suggestions on a first draft of this work. Lastly, we would like to thank Timoth\'e Lemistre for commuticating to us his strategy of proof of Lemma~\ref{lem:embedding}, and to Andrea Seppi for pointing out to us an issue in the proof of a previous version of Lemma~\ref{lem:intorno nel dominio}.

The first author acknowledges support from the European Research Council (ERC) under the European Union’s Horizon 2020 research and innovation programme (grant agreement No 101018839).

\section{Spatial distances}

Let $M$ be a complete, spacelike, maximal $p$-submanifold of $\mb{H}^{p,q}$, as considered in \cite{SST23}. In this section we show that there is an explicit bi-Lipschitz equivalence between the intrinsic Riemannian distance and the pseudo-metric on $M$ induced by $\mb{H}^{p,q}$ (as described in Theorem~\ref{thm:distance comparison}). In order to do so, we study the gradient of the pseudo-Riemannian distance and Busemann functions on $M$ and establish uniform bounds for their $L^\infty$-norms (compare with Theorem~\ref{thm:main3}), defined as
\[
\norm{\nabla \beta}_{M, \infty} : = \sup_{x \in M} \sup_{v \in T^1_x M} \langle \nabla \beta|_x, v \rangle ,
\]
where $\nabla \beta$ denotes the gradient of $\beta : M \to \mathbb{R}$ with respect to the induced metric of $M$.

\subsection{Ricci curvature of spacelike submanifolds}
Let us start by recalling the various notions of curvature tensors for pseudo-Riemannian manifolds. Let $N$ be a smooth manifold endowed with a non-degenerate metric tensor $\langle \bullet, \bullet \rangle$ and let $D$ denote its Levi-Civita connection, namely the unique connection on $T N$ that is compatible with $\langle \bullet, \bullet \rangle$ and that satisfies $[X,Y] = D_X Y - D_Y X$ for any tangent vector fields $X, Y \in \Gamma(T N)$. The Riemann tensors of $D$ of type $(3,1)$ and $(0,4)$ are given by
\begin{gather*}
	R^D(U,V)W := D_V D_U W - D_U D_V W - D_{[V,U]} W , \\
	R^D(U,V,W,Z) := \langle R^D(U,V)W , Z \rangle ,
\end{gather*}
respectively. The \emph{Ricci curvature tensor} $\mathrm{Ric}_N$ is obtained from $R^D$ by taking the trace in its first and third entries. In other words,
\[
\mathrm{Ric}_N(U,V) = \tr_N(R^D(\bullet, U, \bullet, V)) =  \sum_{i = 1}^{\dim N} R^D(e_i,U,e_i,V) ,
\]
where $(e_i)_i$ denotes a local orthonormal frame of $(N,\langle \bullet, \bullet \rangle)$. 

If $(N,\langle \bullet, \bullet \rangle) = (\mathbb{H}^{p,q}, \langle \bullet , \bullet \rangle)$, we have
\begin{equation}\label{eq:riemann hpq}
	R^D(U,V)W = \langle V, W \rangle \, U - \langle U, W \rangle \, V, \qquad R_N(U,V) = - (p + q - 1) \, \langle U, V \rangle .
\end{equation}

If $M$ is a submanifold of $(N,\langle \bullet, \bullet \rangle)$, then its \emph{first fundamental form} is given by the restriction of the metric $\langle \bullet, \bullet \rangle$ on its tangent bundle, and we denote it by $g_M : = \langle \bullet, \bullet \rangle |_{M}$. We say that $(M, g_M)$ is \emph{spacelike} if $g_M$ is a positive definite Riemannian metric. If the first fundamental form of a spacelike submanifold $M$ is geodesically (or metric) complete, we will simply say that $M$ is \emph{complete}. We will be mostly interested in the case of $N = \mathbb{H}^{p,q}$ and when $M$ is a $p$-dimensional spacelike complete submanifold of $\mathbb{H}^{p,q}$.

Recall that the Levi-Civita connection $\nabla$ of $g_M$ coincides with the orthogonal projection of the Levi-Civita connection $D$ on $N$ onto $T M$. The projection of $D$ onto the normal bundle $(T M)^\perp$ defines the \emph{second fundamental form} $\mathbb{I}$ of $M$. In other words, we have
\[
D_U V = \nabla_U V + \mathbb{I}(U,V) ,
\]
for any $U, V \in \Gamma(T M)$. By construction, $\mathbb{I} \in \Gamma((T^* M)^{\otimes 2} \otimes (TM)^\perp)$. Given any normal vector field $n \in \Gamma((TM)^\perp)$, the \emph{shape operator $B_n$} associated with $n$ is defined as the unique endomorphism of $T M$ that satisfies
\[
\langle B_n(U) , V \rangle = \langle \mathbb{I}(U,V) , n \rangle, \qquad U, V \in \Gamma(T M).
\]

The Riemann tensors $R^D$, $R^\nabla$, and the second fundamental form $\mathbb{I}$ are related by the \emph{Gau{\ss}} and \emph{Codazzi equations}
\begin{equation}\label{eq:Gauss-Codazzi}
	\begin{cases}
		R^\nabla(U,V,W,Z) = R^D(U,V,W,Z) + \langle \mathbb{I}(U,W), \mathbb{I}(V,Z) \rangle - \langle \mathbb{I}(U,Z), \mathbb{I}(V,W) \rangle , \\
		(D_U \mathbb{I})(V,W) - (D_V \mathbb{I})(U,W) = R^\nabla(U,V)W - R^D(U,V)W ,
	\end{cases}
\end{equation}
where $U$, $V$, $W$, $Z$ are tangent vector fields to $M$. By taking the trace of the first identity with respect to the first fundamental form of $M$, we deduce that
\[
\mathrm{Ric}_M(U,V) = \tr_M(R^D(\bullet,U,\bullet,V)) + \langle \tr_M \mathbb{I} , \mathbb{I}(U,V) \rangle - \tr_M(\langle \mathbb{I}(\bullet,U) , \mathbb{I}(\bullet,V) \rangle) .
\]
The section $H : = \tfrac{1}{\dim M} \tr_M \mathbb{I}$ is called the \emph{mean curvature vector field} of $M \subset N$. A submanifold (of a Riemannian manifold) with vanishing mean curvature vector field is called \emph{minimal}. 

In the case of $N = \mathbb{H}^{p,q}$ and $M$ a spacelike $p$-submanifold of $\mathbb{H}^{p,q}$ (which will be the case of interest in this paper), the Ricci tensor of $M$ satisfies
\begin{equation}\label{eq:traccia gauss}
	\mathrm{Ric}_M(U,V) = - (p - 1)\, g_M(U,V) + \langle \tr_M \mathbb{I} , \mathbb{I}(U,V) \rangle - \tr_M(\langle \mathbb{I}(\bullet,U) , \mathbb{I}(\bullet,V) \rangle) .
\end{equation}

In this context, spacelike $p$-submanifolds with vanishing mean curvature vector field are usually referred as \emph{maximal}. Notice that, under these hypotheses, the metric $\langle \bullet, \bullet \rangle$ is negative definite on $(T M)^\perp$. In particular, the last term in the right hand-side of \eqref{eq:traccia gauss} is always non-negative and therefore Property \eqref{eq:ricci bound} follows.

\subsection{Fermi charts}

For technical reasons, it is convenient to work on the double cover 
\[
\widehat{\mathbb{H}}^{p,q} \cup \partial \widehat{\mathbb{H}}^{p,q} \longrightarrow \mathbb{H}^{p,q} \cup \partial \mathbb{H}^{p,q} ,
\]
given by
\begin{align*}
	\widehat{\mathbb{H}}^{p,q} & : = \{ x \in \mathbb{R}^{p,q+1} \mid \langle x, x \rangle_{p,q+1} = -1\} , \\
	\partial \widehat{\mathbb{H}}^{p,q} & : = \{z \in  \mathbb{R}^{p,q+1} - \{0\} \mid \langle z, z \rangle_{p,q+1} = 0\}/(z \sim  \lambda \, z , \lambda \in \mathbb{R}^+)
\end{align*}
with its natural projection. The main point is that $\widehat{\mathbb{H}}^{p,q}$ admits the following natural model: Consider a splitting $\mb{R}^{p,q+1}=U\oplus V$ where $U$ is $p$-subspace where the form is positive definite and $V=U^\perp$. Denote by $\mb{D}^p\subset U$ the unit disk and by $\mb{S}^q\subset V$ the $(-1)$-sphere. Then we have a diffeomorphism
\[
\begin{matrix}
	\mathbb{D}^p \times \mathbb{S}^q & \longrightarrow & \widehat{\mathbb{H}}^{p,q} \\
	(x,y) & \longmapsto & \frac{2}{1 - \norm{x}^2} \, x + \frac{1 + \norm{x}^2}{1 - \norm{x}^2} \, y ,
\end{matrix}
\]
which continuously extends to the ideal boundary $\partial \widehat{\mathbb{H}}^{p,q}$ as
\[
\begin{matrix}
	\partial \mathbb{D}^p \times \mathbb{S}^q & \longrightarrow & \partial \widehat{\mathbb{H}}^{p,q} \\
	(x,y) & \longmapsto & [x + y] .
\end{matrix}
\]

In these coordinates, every complete spacelike $p$-submanifold $\widehat{M} \subset \widehat{\mathbb{H}}^{p,q}$ can be expressed as the graph of some function $u : \mathbb{D}^p \to \mathbb{S}^q$ that is strictly $1$-Lipschitz with respect to the hemispherical metric of $\mathbb{D}^{p} \subset \mathbb{S}^p$ and the spherical metric of $\mathbb{S}^q$. Furthermore, the ideal boundary $\partial \widehat{M} \subset \partial \mathbb{H}^{p,q} \cong \partial \mathbb{D}^p \times \mathbb{S}^q$ is the graph of the continuous extension of $u$ to $\partial \mathbb{D}^p$. For details, see \cite{SST23}*{\S~3.2}.

We also recall that the projection $\pi:\widehat{\mb{H}}^{p,q}\to\mb{D}^p$ satisfies 
\begin{equation}
\label{eq:differential}
g_{\mb{H}^p}(\dd\pi_x(v),\dd\pi_x(v))\ge\langle v,v\rangle
\end{equation}
for every $x\in\widehat{\mb{H}}^{p,q}$ and $v\in T_x\widehat{\mb{H}}^{p,q}$. Using this fact one shows the following standard estimate on the Riemannian distance on $\widehat{M}$. 

\begin{lem}
\label{lem:distance}
We have $d_{\widehat{M}}(x,y)\le {\rm arccosh}(-\langle x,y\rangle)=:d_{\mb{H}^{p,q}}(x,y)$.
\end{lem}

\begin{proof}
Let $U$ be a positive $p$-subspace in $T_x\widehat{\mb{H}}^{p,q}=x^\perp$ containing the initial velocity $v$ of the spacelike geodesic in $\widehat{\mb{H}}^{p,q}$ joining $x$ to $y$. We work in Fermi charts $\mb{D}^p\times\mb{S}^q\to\widehat{\mb{H}}^{p,q}$ adapted to $U$ as described above. In these coordinates $\widehat{M}$ can be written as the graph of a smooth function $u:\mb{D}^p\to\mb{S}^q$ which is 1-Lipschitz with respect to the (hemi)spherical metrics on source and target. Furthermore, a direct computation shows that $\pi(x)=0$ and $\pi(y)=\tanh(t)v$ where $y=\cosh(t)x+\sinh(t)v$ and $\cosh(t)=-\langle x,y\rangle$. Let $\gamma:I\to \mb{D}^p$ be the hyperbolic geodesic joining $\pi(x)$ to $\pi(y)$. By the formula for the hyperbolic metric on $\mb{D}^p$ and the expressions for $\pi(x)$ and $\pi(y)$ we compute
\[
d_{\mb{H}^p}(\pi(x),\pi(y))={\rm arccosh}(-\langle x,y\rangle).
\]
Consider the path $\alpha=(\gamma,u\gamma)$ on $\widehat{M}$. By our choices $\alpha$ joins $x$ to $y$ and, hence, has length greater than $d_{\widehat{M}}(x,y)$. On the other hand, by equation \eqref{eq:differential}, we have
\[
d_{\mb{H}^p}(\pi(x),\pi(y))=\ell_{\mb{H}^p}(\gamma)=\ell_{\mb{H}^p}(\pi\alpha)\ge\ell_{\mb{H}^{p,q}}(\alpha)
\]
which finishes the proof.
\end{proof}

\subsection{Distance and Busemann functions}
We briefly recall the notions of pseudo-Riemannian Busemann and distance functions in $\widehat{\mathbb{H}}^{p,q}$ and $\mathbb{H}^{p,q}$.

\begin{dfn}[Busemann Function]
	Let $\theta \in \widehat{\mathbb{H}}^{p,q}$ and $o \in \{x \in \widehat{\mathbb{H}}^{p,q} \mid \langle x, \theta \rangle < 0 \}$. The pseudo-Riemannian \emph{Busemann function} (in $\widehat{\mathbb{H}}^{p,q}$) associated with $\theta$ and based at $o$ is given by
	\[
	b_{\theta, o} : \{x \in \widehat{\mathbb{H}}^{p,q} \mid \langle x, \theta \rangle < 0 \} \to \mathbb{R}, \qquad b_{\theta, o}(x) : = \log\frac{\langle x, \theta \rangle}{\langle o, \theta\rangle} .
	\]
	(Here we identify, with abuse, the point $\theta$ with one of its representatives in $\mathbb{R}^{p,q+1}$. Notice that the definition of $b_{\theta, o}$ is independent of this choice.) 
	
	Similarly, for any $\theta' \in \partial \mathbb{H}^{p,q}$ and $o'  \in \mathbb{H}^{p,q} - \mathbb{P}((\theta')^\perp)$, the pseudo-Riemannian \emph{Busemann function} (in $\mathbb{H}^{p,q}$) associated with $\theta'$ and based at $o'$ is given by
	\[
	b_{\theta',o'} : \mathbb{H}^{p,q} - \mathbb{P}((\theta')^\perp) \to \mathbb{R}, \qquad  b_{\theta',o'}(x') : = b_{\theta, o}(x) ,
	\]
	where $o, x \in \widehat{\mathbb{H}}^{p,q}, \theta \in \partial\widehat{\mathbb{H}}^{p,q}$ are representatives of $o', x' \in \mathbb{H}^{p,q}, \theta' \in \partial \mathbb{H}^{p,q}$ satisfying $\langle o ,\theta \rangle < 0$ and $\langle x ,\theta \rangle < 0$. 
\end{dfn}

\begin{dfn}[Distance Function]
	Let $o \in \widehat{\mathbb{H}}^{p,q}$. The pseudo-Riemannian \emph{distance function} (in $\widehat{\mathbb{H}}^{p,q}$) from $o$ is given by
	\[
	d_{\widehat{\mathbb{H}}^{p,q}}(o,\bullet) : \{ x \in \widehat{\mathbb{H}}^{p,q} \mid \langle o, x \rangle < -1 \} \to \mathbb{R}, \qquad  d_{\mathbb{H}^{p,q}}(o,x) : = \mathrm{arccosh}(- \langle o, x \rangle) .
	\]
	
	Similarly, for any $o' \in \mathbb{H}^{p,q}$, the pseudo-Riemannian \emph{distance function} (in $\mathbb{H}^{p,q}$) from $o'$ is given by
	\[
	d_{\mathbb{H}^{p,q}}(o',\bullet) : \{ x' \in \mathbb{H}^{p,q} \mid \abs{\langle o', x' \rangle} > 1 \} \to \mathbb{R}, \qquad d_{\mathbb{H}^{p,q}}(o',x') : =  d_{\widehat{\mathbb{H}}^{p,q}}(o,x) ,
	\]
	where $o, x \in \widehat{\mathbb{H}}^{p,q}$ are representatives of $o', x' \in \mathbb{H}^{p,q}$ satisfying $\langle o ,x \rangle < - 1$.
\end{dfn}

For future convenience, we describe the Hessian and the Laplacian of the restrictions of Busemann functions on maximal submanifolds.

\begin{pro}\label{prop:hessian of busemann}
	Let $M$ be a spacelike $p$-submanifold of $\widehat{\mathbb{H}}^{p,q}$ and let $\theta \in \partial \widehat{\mathbb{H}}^{p,q}$ be such that $\langle \bullet ,\theta \rangle < 0$ on $M$. Then the Hessian $\nabla^2 b_{\theta,o}$ with respect to the Levi-Connection of the induced metric $g_M$ of $M$ satisfies
	\[
	\nabla^2 b_{\theta, o} = g_M - \dd{ b_{\theta,o}} \otimes \dd{ b_{\theta,o}} + \dd b_{\theta, o} (\mathbb{I}(\bullet, \bullet)) ,
	\]
	In particular, when $M$ is a spacelike maximal $p$-submanifold, $b_{\theta, o}|_M$ satisfies
	\[
	\Delta b_{\theta, o} = p - \norm{\nabla b_{\theta,o}}^2 ,
	\]
	where $\nabla b_{\theta, o}$ denotes the gradient of $b_{\theta, o}|_M$ with respect to $g_M$ and $\Delta u := \tr_M (\nabla^2 u)$.
\end{pro}

\begin{proof}
	Let $D^0$, $D$, and $\nabla$ be the Levi-Civita connections of $\mathbb{R}^{p,q+1}$, $\widehat{\mathbb{H}}^{p,q}$, and $M$, respectively. Given any two tangent vector fields $V,W$ on $M$, their covariant derivatives at $y \in M$ are related by the identities:
	\begin{align}\label{eq:identities connections}
		\begin{split}
			D^0_V W|_y & = D_V W|_y + \langle V, W \rangle|_y \, y \in T_y \widehat{\mathbb{H}}^{p,q} \oplus (T_y \widehat{\mathbb{H}}^{p,q})^\perp , \\
			D_V W|_y & = \nabla_V W|_y + \mathbb{I}(V,W)|_y \in T_y M \oplus ((T_y M)^\perp \cap T_y \widehat{\mathbb{H}}^{p,q}) .
		\end{split}
	\end{align}
	Let $\mathrm{Hess}$ and $\nabla^2$ denote the Hessians with respect to the connections $D$ and $\nabla$, respectively. Then we have:
	\begin{align*}
		(\nabla^2 b_{\theta,o})(V,W) & = (\nabla_V \dd b_{\theta,o})(W) \\
		& = V(W(b_{\theta,o})) - \dd{b_{\theta,o}}(\nabla_V W) \\
		& = (\mathrm{Hess}\, b_{\theta,o})(V,W) + \dd{b_{\theta,o}}(D_V W - \nabla_V W) \\
		& = (\mathrm{Hess}\, b_{\theta,o})(V,W) + \dd{b_{\theta,o}}(\mathbb{I}(V,W)) .
	\end{align*}
	We can express the term $(\mathrm{Hess}\, b_{\theta,o})(V,W)$ as follows:
	\begin{align*}
		(\mathrm{Hess}\, b_{\theta,o})_x(v,w) & = (D_v \dd b_{\theta, o})_x(w) \\
		& = V\left( \frac{\langle W , \theta \rangle}{\langle \mathrm{id} , \theta \rangle}\right)|_x - \frac{\langle D_v W|_x, \theta \rangle}{\langle x ,\theta \rangle} \\
		& = \frac{\langle D^0_v W , \theta \rangle}{\langle x, \theta \rangle} - \frac{\langle v , \theta \rangle \, \langle w , \theta \rangle}{\langle x, \theta \rangle^2} - \frac{\langle D_v W|_x, \theta \rangle}{\langle x ,\theta \rangle} \\
		& = \langle v, w \rangle - \frac{\langle v , \theta \rangle \, \langle w , \theta \rangle}{\langle x, \theta \rangle^2} ,
	\end{align*}
	where in the last step we applied \eqref{eq:identities connections}. Combining the relations we found, we get:
	\[
	(\nabla^2 b_{\theta, o})_x(v,w) = \langle v, w \rangle - \frac{\langle v , \theta \rangle \, \langle w , \theta \rangle}{\langle x, \theta \rangle^2} + \frac{\langle \mathbb{I}_x(v,w), \theta \rangle}{\langle x, \theta \rangle} ,
	\]
	which coincides with the desired identity. By taking the trace with respect to $g_M$ and recalling that $\tr_M \mathbb{I} = 0$ when $M$ is maximal, the second part of the statement follows.
\end{proof}

\subsection{A compactness criterion}
In the proof of Theorem~\ref{thm:main3} we will study the restrictions of pseudo-Riemannian distance and Busemann functions at the points where their gradients have maximal norm. In order to carry out this strategy we need to ensure the existence of those points. The compactness criterion that we describe here will take care of this issue. 

Following \cites{LTW,SST23}, we consider the space of \emph{pointed complete, spacelike, maximal $p$-submanifolds of $\widehat{\mathbb{H}}^{p,q}$}
\[
\mathcal{M}\!\mathit{ax}_{p,q} : = \left\{ (x,M) \left| 
\begin{array}{c}
	\text{$M$ complete, spacelike, maximal} \\
	\text{$p$-submanifold of $\widehat{\mathbb{H}}^{p,q}$, $x \in M$}
\end{array}
\right. \right\}
\]
and the disk bundle
\[
\mathcal{BM}\!\mathit{ax}_{p,q} : = \left\{ (x,M,z) \left| 
\begin{array}{c}
	\text{$M$ complete, spacelike, maximal} \\
	\text{$p$-submanifold of $\widehat{\mathbb{H}}^{p,q}$, $x \in M$ and $z \in M \cup \partial M$}
\end{array}
\right. \right\} .
\]
We endow $\mathcal{M}\!\mathit{ax}_{p,q}$ and $\mathcal{BM}\!\mathit{ax}_{p,q}$ with the restriction of the Hausdorff topology on pointed entire graphs inside $\widehat{\mathbb{H}}^{p,q}$ and with the subspace topology from the inclusion $\mathcal{BM}\!\mathit{ax}_{p,q} \subset \mathcal{M}\!\mathit{ax}_{p,q} \times (\widehat{\mathbb{H}}^{p,q} \cup \partial\widehat{\mathbb{H}}^{p,q})$, respectively. Then we prove:

\begin{pro}\label{prop:compactness maximal}
	The action of $\mathrm{SO}(p,q+1)$ on $\mathcal{BM}\!\mathit{ax}_{p,q}$ is cocompact.
\end{pro}

This is essentially a corollary of the following compactness criterion due to Seppi, Smith, and Toulisse \cite{SST23}.

\begin{thm}[{\cite{SST23}*{Theorem~5.3}}]\label{thm:compactness maximal}
	Let
	\[
	\mathrm{Gr}_p(\widehat{\mathbb{H}}^{p,q}) : = \{(x,H) \mid 
	\text{$x \in \widehat{\mathbb{H}}^{p,q}$, $H \subset T_x \widehat{\mathbb{H}}^{p,q}$ spacelike $p$-plane } \} .
	\]
	Then the map 
	\[
	\begin{matrix}
		\tau : & \mathcal{M}\!\mathit{ax}_{p,q} & \longrightarrow & \mathrm{Gr}_p(\widehat{\mathbb{H}}^{p,q}) \\
		& (x,M) & \longmapsto & (x,T_x M) ,
	\end{matrix}
	\]
	is proper.
\end{thm}

\begin{proof}[Proof of Proposition~\ref{prop:compactness maximal}]
	Let $((x_n,M_n,z_n))_n$ be a sequence in $\mathcal{BM}\!\mathit{ax}_{p,q}$. Choose $x \in \widehat{\mathbb{H}}^{p,q}$ a basepoint and $H \subset T_{x} \widehat{\mathbb{H}}^{p,q}$ a fixed spacelike $p$-plane at $x$. Then there exist elements $(g_n)_n$ in $\mathrm{SO}(p,q+1)$ such that $g_n \cdot x_n = x$ and $(\dd{g_n})_{x_n}(T_{x_n} M_n) = H$. By Theorem~\ref{thm:compactness maximal}, up to subsequence $((g_n \cdot x_n, g_n \cdot M_n))_n$ converges in the Hausdorff topology to some pointed, complete, spacelike, maximal $p$-submanifold $(x,M)$. 
	
	It remains to show that, up to subsequence, $(g_n \cdot z_n)_n$ converges to some point $z \in M \cup \partial M$. Notice that, if the sequence $(g_n \cdot z_n)_n$ stays in a compact subset of $\widehat{\mathbb{H}}^{p,q}$ then, by definition of Hausdorff convergence, every limit point of the sequence lies inside the Hausdorff limit of $(g_n \cdot M_n)_n$, which is equal to $M$. Therefore the only issue is given by sequences that leave every compact subset. We need to check that all of them subconverge to a point in $\partial M \subset \partial \widehat{\mathbb{H}}^{p,q}$.
	
	To see this, recall that every $g_n \cdot (M_n \cup \partial M_n)$ is the graph of some $1$-Lipschitz function $\mathbb{D}^p \cup \partial \mathbb{D}^p \to \mathbb{S}^q$ with respect to the (hemi)spherical metrics of source and target, and that $g_n \cdot x_n = x$. By Arzel\'a-Ascoli, the space of such graphs is compact in the Hausdorff topology and, hence, we conclude that $(g_n \cdot z_n)_n$ converges to some $z \in \partial M$.
\end{proof}

\subsection{Lower bound for the gradient}
Before moving to the proof of Theorem~\ref{thm:main3}, let us observe the following issue: In the Riemannian setting, the norm of the gradient of the restriction of a Busemann function to a submanifold is always $\le 1$. However, this is not true in $\mb{H}^{p,q}$. In fact we have the following:

\begin{pro}
\label{pro:busemann}

Let $M\subset\mb{H}^{p,q}$ be a complete, spacelike $p$-submanifold with boundary at infinity $\partial M \subset \partial\mb{H}^{p,q}$ and let $o \in M$ be a basepoint. Denote by $\nabla$ the Levi-Civita connection of $M$. If
\[
\sup_{\theta\in\Lambda} \norm{\nabla b_{\theta,o}}_{M, \infty}\le 1 \text{ or } \sup_{o \in M} \norm{\nabla (d_{\mathbb{H}^{p,q}}(o,\bullet))}_{M - \{o\}, \infty} \leq 1,
\]
then $M$ is a totally geodesic copy of $\mb{H}^p$.

\end{pro}

\begin{proof}
	Being $M \cup \partial M \subset \mathbb{H}^{p,q} \cup \partial \mathbb{H}^{p,q}$ homeomorphic to a closed disk, it admits a lift to $\widehat{\mathbb{H}}^{p,q} \cup \partial \widehat{\mathbb{H}}^{p,q}$, which determines representatives of $o, \theta$ inside $\widehat{\mathbb{H}}^{p,q} \cup \partial \widehat{\mathbb{H}}^{p,q}$. With abuse, we continue to denote these objects by $M \cup \partial M$, $o$, $\theta$. 
	
	As in Labourie and Toulisse \cite{LT22}*{\S~5.2}, it will be convenient to use a unified notation for both pseudo-Riemannian distance and Busemann functions. In what follows, we let $z \in M \cup \partial M$ and we denote by $\beta = \beta_z$ either the distance function $\beta = d_{\widehat{\mathbb{H}}^{p,q}}(o,\bullet)$ when $z = o \in M$, or the Busemann function $\beta = b_{\theta,o}$ when $z = \theta \in \partial M$. In the latter case, we will not distinguish between the positive projective class $z$ and a choice of a representative in it, as all expressions that follow turn out to be independent of such choice.
	
	The differential of $\beta = \beta_z$ at any $x \in \widehat{\mathbb{H}}^{p,q}$ can be expressed as
	\begin{equation} \label{eq:formula_differential}
		(\dd{\beta})_x(u) = - \frac{\langle u, z \rangle}{\sqrt{\langle z, z \rangle + \langle x, z \rangle^2}} .
	\end{equation}
	It follows from \eqref{eq:formula_differential} that the gradient of $\beta$ with respect to the metric $g_{\widehat{\mathbb{H}}^{p,q}}$ satisfies
\begin{equation}\label{eq:gradient busemann}
	\mathrm{grad} \, \beta|_x = \frac{1}{\sqrt{\langle z, z \rangle + \langle x, z \rangle^2}} \, (z + \langle x, z \rangle \, x) ,
\end{equation}
for any $x \in M - \{z\}$. As we restrict $\beta$ to $M - \{z\}$, its gradient $\nabla \beta$ with respect to $g_M$ coincides with the orthogonal projection of $\mathrm{grad} \, \beta|_x$ onto $T_x M$. 

Consider $w := (\mathrm{grad} \, \beta - \nabla \beta)|_x \in T_x \widehat{\mathbb{H}}^{p,q} \cap (T_x M)^\perp$. By construction $\langle w,w \rangle \leq 0$, so  
\[
\norm{\mathrm{grad} \, \beta}_x^2  = \norm{\nabla \beta}_x^2 + \langle w,w \rangle \leq \norm{\nabla \beta}_x^2 \leq 1 ,
\]
where in the last step we applied our hypothesis. On the other hand,
\begin{align*}
	\norm{\mathrm{grad} \, \beta}_x^2 = \frac{\langle z, z \rangle - \langle x, z \rangle^2 + 2 \, \langle x, z \rangle^2}{\langle z, z \rangle + \langle x, z \rangle^2} = 1 .
\end{align*}
Hence, for the inequality above to hold, we must have $w = 0$ or, in other words, $\mathrm{grad} \, \beta|_x = \nabla \beta|_x$.

If $z = \theta \in \partial M$ and $\beta = b_{\theta,o}$, then this argument applies to every $\theta \in \partial M$, so combining the identity $\mathrm{grad} \, b_{\theta,o}|_x = \nabla b_{\theta,o}|_x$ with \eqref{eq:gradient busemann}, we deduce that 
\[
\partial M \subset (\mathrm{Span}(x) + T_x M) \cap \partial \widehat{\mathbb{H}}^{p,q} .
\]

Since $\partial M$ is a topological $(p-1)$-sphere embedded in $\partial \widehat{\mathbb{H}}^{p,q}$, the vector subspace ${\rm Span}(\partial M) \subset \mathbb{R}^{p,q+1}$ has dimension at least $p+1={\rm dim}(\mathrm{Span}(x) + T_x M)$. It follows that ${\rm Span}(\partial M)=\mathrm{Span}(x) +T_x M$. Since the point $x \in M$ is arbitrary, we deduce that the totally geodesic spacelike subspace $\widehat{\mb{H}}^p_x$ given by the connected component of 
\[
(\mathrm{Span}(x)+T_x M)\cap\widehat{\mb{H}}^{p,q}={\rm Span}(\partial M)\cap\widehat{\mb{H}}^{p,q}
\]
that contains $x$ is in fact independent of $x$ and $M=\widehat{\mb{H}}^p_x$, being $M$ complete.

Similarly, if $\beta = d_{\widehat{\mathbb{H}}^{p,q}}(o,\bullet)$, the argument from above, together with \eqref{eq:gradient busemann}, implies that $o \in \mathrm{Span}(x) + T_x M$ for every $o \in M - \{x\}$. In particular, we deduce that $M \subset \widehat{\mb{H}}^p_x$, where $\widehat{\mb{H}}^p_x$ is defined as above. Again by the completeness of $M$, it follows that $M = \widehat{\mathbb{H}}^{p}_x$.
\end{proof}

This proves the rigidity part (1) in Theorem~\ref{thm:main3}.

\subsection{Upper bound for the gradient}
We now move to the proof of Theorem~\ref{thm:main3}. 

\begin{thm}[{Theorem~\ref{thm:main3}}]
	Let $M$ be a complete, spacelike, maximal $p$-submanifold of $\mathbb{H}^{p,q}$ and let $\beta:M\to\mb{R}$ denote the restriction of either the pseudo-Riemaniann distance from some $o\in M$ or the pseudo-Riemannian Busemann function associated with some $\theta\in\partial M \subset \partial \mathbb{H}^{p,q}$. Then
	\[
	\norm{\nabla \beta}_{M,\infty} \leq \sqrt{p} .
	\] 
\end{thm}

\begin{proof}[Proof of Theorem~\ref{thm:main3}]
	
	As in Lemma~\ref{lem:critical point}, we select a lift of $M \cup \partial M$ to $\widehat{\mathbb{H}}^{p,q} \cup \partial \widehat{\mathbb{H}}^{p,q}$ and adopt the same conventions (and abuses of notations). Given any $z \in M \cup \partial M$, we denote by $\beta = \beta_z : M \to \mathbb{R}$ either the pseudo-Riemannian Busemann or distance function associated with $z$, depending on whether $z = o \in M \subset \widehat{\mathbb{H}}^{p,q}$ or $z = \theta \in \partial \widehat{\mathbb{H}}^{p,q}$, respectively. The norm of the gradient of $\beta$ with respect to the induced metric of $M$ can be characterized as
	\[
	\norm{\nabla \beta}_x = \max_{u \in T_x^1 M} (\dd\beta)_x(u) .
	\]
	
	Now we define
	\[
	F_M : T^1 M \times (M \cup \partial M) \longrightarrow \mathbb{R}, \qquad F_M(x,u,z) : = (\dd\beta_z)_x(u) = \langle \nabla \beta_z|_x, u \rangle .
	\]
	It is not hard to see that the function $F_M$ is continuous on $T^1 M \times (M \cup \partial M)$ (see e.g. \eqref{eq:formula_differential}). We claim that it is not restrictive to assume that $F_M$ achieves its maximum at some point $(x,u,z)$. We deduce this from the compactness criterion of Proposition~\ref{prop:compactness maximal}. More precisely, consider
	\[
	\begin{matrix}
		\mathcal{F} : \mathcal{BM}\!\mathit{ax}_{p,q} & \longrightarrow & \mathbb{R} \\
		(x,M,z) & \longmapsto & \norm{\nabla \beta_z}_x .
	\end{matrix}
	\]
	The map $\mathcal{F}$ is continuous and $\mathrm{SO}(p,q+1)$-invariant. Therefore, by Proposition~\ref{prop:compactness maximal}, $\mathcal{F}$ attains a maximum and, by definition, $\max \mathcal{F} \geq \sup F_M$ for any complete, spacelike maximal $p$-submanifold $M$ of $\widehat{\mathbb{H}}^{p,q}$. Throughout the remainder of the exposition, we denote by $(x,M,z) \in \mathcal{BM}\!\mathit{ax}_{p,q}$ a point realizing $\max \mathcal{F}$ and we set $F = F_M$, $\beta = \beta_z$, and $u \in T_x^1 M$ to be a unit vector such that $\norm{\nabla \beta_z}_x = \langle \nabla \beta|_x , u \rangle$.
	
	Notice that, if $M = \mathsf{M}_\rho$ is invariant under the action of some convex cocompact representation $\rho : \Gamma \to \mathrm{SO}(p,q+1)$, then we do not need to apply the compactness criterion from Proposition~\ref{prop:compactness maximal} to conclude the existence of a maximum. Indeed, the map $F = F_M$ is invariant under the action of $\Gamma$ given by
	\[
	\gamma \cdot (x,u,z) = (\rho(\gamma)(x), (\dd\rho(\gamma))_x(u), \rho(\gamma)(z)) , \qquad \gamma \in \Gamma, (x,u,z) \in T^1 M \times (M \cup \partial M) .
	\]	
	Moreover, the quotient $(T^1 M \times (M \cup \partial M))/\Gamma$, being the total space of a bundle over $M/\rho(\Gamma)$ with fiber $\mathbb{S}^{p - 1} \times (M \cup \partial M)$, is compact and, hence, $F$ attains its maximum at some point $(x,u,z)$. 
	
	We now apply the following technical lemma, whose proof will be postponed to the end of this argument.

	\begin{lem}\label{lem:maximum principle}
		Let $M$ be a complete, spacelike maximal $p$-submanifold of $\widehat{\mathbb{H}}^{p,q}$. If $(x,u) \in T^1 M$ is a point of maximum of $F(\bullet,z)$ and $L : = F(x,u,z)$, then
		\[
		L^2 - p - \frac{\langle z, z \rangle}{\langle u, z \rangle^2} \, L^2 (2 L^2 - p - 1) \leq 0 ,
		\]
		where $z$ is either a point in $M \subset \widehat{\mathbb{H}}^{p,q}$ or (a representative of the positive projective class of) a point in $\partial M \subset \partial \widehat{\mathbb{H}}^{p,q}$.
	\end{lem}

	If $z \in \partial M$, then $\langle z, z \rangle = 0$. In particular, we deduce from Lemma~\ref{lem:maximum principle} that $L^2 \leq p$. On the other hand, if $z \in M$ and $d : = d_{\mathbb{H}^{p,q}}(z,x)$, then $\langle z, z \rangle = -1$, $\cosh d = - \langle x, z \rangle$ and
	\[
	L^2 = \frac{\langle u , z \rangle^2}{\sinh^2 d} .
	\]
	Using these identities we see that the inequality from Lemma~\ref{lem:maximum principle} can be rewritten as
	\[
	\frac{\cosh^2 d + 1}{\sinh^2 d} \left(L^2 - p + \frac{p - 1}{\cosh^2 d + 1}\right) \leq 0 .
	\]
	It follows that also in this case $L^2 < p$.
\end{proof}

For future convenience, we observe that Theorem~\ref{thm:main3} can be restated in the following way.

\begin{thm}\label{thm:restatement_thm:main3}
	Let $M$ be a complete, spacelike, maximal $p$-submanifold of $\mathbb{H}^{p,q}$. For any $v \in \mathbb{R}^{p,q+1}$ and $y \in M$, we denote by $v^N_y$ the orthogonal projection of $v$ onto $T_y \mathbb{H}^{p,q} \cap (T_y M)^\perp$. Then:
	\begin{enumerate}
		\item For any $x, y \in M$, 
		\[
		0 \leq - \langle x^N_y , x^N_y \rangle \leq (p - 1) \, \sinh^2 d_{\mathbb{H}^{p,q}}(x,y) .
		\]
		\item For any $\theta \in \partial M \subset \partial \mathbb{H}^{p,q}$ and for any $y \in M$,
		\[
		0 \leq - \langle \tilde{\theta}^N_y, \tilde{\theta}^N_y \rangle \leq (p - 1) \, \langle \tilde{\theta}, y \rangle^2 
		\]
		where $\tilde{\theta} \in \mathbb{R}^{p,q+1} - \{0\}$ is some (any) representative of $\theta \in \partial \mathbb{H}^{p,q}$.
	\end{enumerate}
\end{thm}

\begin{proof}
	With the same notations of Proposition~\ref{pro:busemann}, we denote by $\beta = \beta_z : M \to \mathbb{R}$ the restriction of either the Busemann function $b_{\theta,o}$ (for some choice of basepoint $o \in M$) associated with $z = \theta \in \partial M$, or the distance function $d_{\mathbb{H}^{p,q}}(x,\cdot)$ from some fixed point $z = x \in M$. 
	
	We observe that, if $\nabla \beta|_y$ denotes the gradient of the restriction of $\beta$ on $M$ with respect to its induced metric, then $\nabla \beta|_y$ coincides with the projection of $\mathrm{grad} \, \beta$ onto $T_y M$ for any $y \in M$. Since $\norm{\mathrm{grad} \, \beta}^2 = 1$, relation \eqref{eq:gradient busemann} implies:
	\begin{align*}
		- \langle z_y^N, z_y^N \rangle & = - \langle (z + \langle z, y \rangle \, y)_y^N, (z + \langle z, y \rangle \, y)_y^N \rangle \\
		& = (\langle z,z \rangle + \langle z, y \rangle^2) (\norm{\nabla \beta}_y^2 - \norm{\mathrm{grad}\, \beta}_y^2) \\
		& = (\langle z,z \rangle + \langle z, y \rangle^2)(\norm{\nabla \beta}_y^2 - 1) .
	\end{align*}
	By Theorem~\ref{thm:main3}, the term $\norm{\nabla \beta}_y^2$ is bounded from above by $p$. From here, both items follow simply by replacing the role of $z$ with either $x \in M$ or $\theta \in \partial M$ (and noticing that $\sinh d_{\mathbb{H}^{p,q}}(x,y) = \sqrt{\langle x, y \rangle^2 - 1}$ for any pair of space-related points $x, y$). 
\end{proof}

We dedicate the remainder of the current section to the proof of Lemma~\ref{lem:maximum principle}. As above, we select a lift of $M \cup \partial M$ inside $\widehat{\mathbb{H}}^{p,q} \cup \partial \widehat{\mathbb{H}}^{p,q}$, which we continue to denote with abuse by $M \cup \partial M$. Moreover, we let $\beta = \beta_z : M \to \mathbb{R}$ denote the restriction on $M$ of either the pseudo-Riemannian Busemann function $b_{\theta,o}$ in $\widehat{\mathbb{H}}^{p,q}$ when $z = \theta \in \partial M \subset \partial \widehat{\mathbb{H}}^{p,q}$, or a pseudo-Riemannian distance function $\beta = d_{\widehat{\mathbb{H}}^{p,q}}(o,\bullet)$ when $z = o \in M \subset \widehat{\mathbb{H}}^{p,q}$. Finally, we recall
\[
F : T^1 M \times (M \cup \partial M) \longrightarrow \mathbb{R}, \qquad F(x,u,z) : = (\dd\beta_z)_x(u) .
\]

\begin{lem}\label{lem:critical point}
	A point $(x,u) \in T^1 M$ is a critical point for $F_z = F(\bullet,z)$ if and only if 
	\begin{itemize}
		\item $z^T = \langle u, z \rangle \, u$, and
		\item $u$ is an eigenvector of the shape operator $(B_{z^N})_x$,
	\end{itemize}
	where $z^T$ and $z^N$ denote the orthogonal projections of $z$ onto $T_x M$ and $(T_x M)^\perp \cap T_x \widehat{\mathbb{H}}^{p,q} \subset \mathbb{R}^{p,q+1}$, respectively. Moreover, if this happens then
	\[
	(B_{z^N})_x(u) = \langle x, z \rangle \, (F_z(x,u)^2 -1) \, u .
	\]
\end{lem}

\begin{proof}
	As in the proof of Proposition~\ref{prop:hessian of busemann}, we denote by $D^0$, $D$, and $\nabla$ the Levi-Civita connections of $\mathbb{R}^{p,q+1}$, $\widehat{\mathbb{H}}^{p,q}$, and $M$, respectively. To prove the desired statement, we need to compute the differential of $F$ at $(x,u) \in T^1 M$. Recall that the connection $\nabla$ determines a natural decomposition of the tangent space $T_{(x,u)} T^1 M$ into its horizontal and vertical subspaces
	\[
	T_{(x,u)} T^1 M = T^h_{(x,u)} T^1 M \oplus T^v_{(x,u)} T^1 M \cong T_x M \oplus (u^\perp \cap T_x M) .
	\]
	
	To compute the first order variation of $F$ along $(\dot{x}, \dot{u}) \in T_{(x,u)} T^1 M$, we consider a path $\gamma : (-\varepsilon,\varepsilon) \to M$ of $M$ that satisfies $\gamma(0) = x$, $\gamma'(0) = \dot{x} \in T_x M$, and denote by $U, V \in \Gamma(\gamma^*(T^1 M))$ the $\nabla$-parallel vector fields along $\gamma$ that verify $U(0) = u \in T^1_x M$, $V(0) = \dot{u} \in u^\perp \cap T_x M$. We also define
	\[
	W \in \Gamma(\gamma^*(T^1 M)), \quad W(t) : = \frac{U(t) + t \, V(t)}{\norm{U(t) + t\, V(t)}} .
	\]
	Notice that $W(0) = u$ and $\nabla_{\partial_t} W|_{t = 0} = \dot{u} - \frac{\langle u, \dot{u} \rangle}{\norm{u}^3} = \dot{u}$. By construction and by \eqref{eq:identities connections}, we have
	\begin{equation}\label{eq:covariant derivative W}
		D^0_{\partial_t} W|_{t = 0} = \nabla_{\partial_t} W|_{t = 0} +  \mathbb{I}_x(\dot{x}, u) + \langle \dot{x}, u \rangle \, x = \dot{u} +  \mathbb{I}_x(\dot{x}, u) + \langle \dot{x}, u \rangle \, x .
	\end{equation}
	From identity \eqref{eq:formula_differential}, we deduce
	\begin{align}\label{eq:differential1}
		\begin{split}
			(\dd F_z)_{(x,u)}(\dot{x},\dot{u}) & = F_z(\gamma,W)' \,|_{t = 0} \\
			& = \left[ - \frac{\langle D^0_{\partial_t} W , z \rangle}{(\langle z, z \rangle + \langle \gamma, z \rangle^2)^{1/2}} + \frac{\langle W , z \rangle \, \langle \gamma , z \rangle \, \langle \gamma' , z \rangle}{(\langle z, z \rangle + \langle \gamma, z \rangle^2)^{3/2}} \right]_{t = 0}
		\end{split}
	\end{align}
	
	For future convenience we notice that, whenever $\langle u, z \rangle \neq 0$, the expression above can be rewritten as
	\begin{equation}\label{eq:first order derivative}
		(\dd F_z)_{(x,u)}(\dot{x},\dot{u}) = \left[ F_z(\gamma,W) \, \frac{\langle D^0_{\partial_t} W , z \rangle}{\langle W, z \rangle} - F_z(\gamma,W)^3 \frac{\langle \gamma , z \rangle \, \langle \gamma' , z \rangle}{\langle W,z \rangle^2} \right]_{t = 0} .
	\end{equation}

	Applying \eqref{eq:differential1}, we express the differential of $F$ at $(x,u)$ as follows:
	\begin{align}\label{eq:differential2}
		\begin{split}
			(\dd F)_{(x,u)}(\dot{x},\dot{u}) & = - \frac{\langle \dot{u} +  \mathbb{I}_x(\dot{x}, u) + \langle \dot{x}, u \rangle \, x , z \rangle}{(\langle z, z \rangle + \langle x, z \rangle^2)^{1/2}} + \frac{\langle u, z \rangle \, \langle x, z \rangle \, \langle \dot{x}, z \rangle}{(\langle z, z \rangle + \langle x, z \rangle^2)^{3/2}} \\
			& = \phi_z(x)\left( - \langle \dot{u} , z^T \rangle + \langle \dot{x}, \, \phi_z(x) \, F_z(x,u) \, \langle x, z \rangle \, z^T - (B_{z^N})_x(u) - \langle x, z \rangle \, u \rangle \right) ,
		\end{split}
	\end{align}
	where $\phi_z(x) := (\langle z, z \rangle + \langle x, z \rangle^2)^{-1/2}$ and
	\[
	z = z^T + z^N - \langle x, z \rangle \, x \in T_x M \oplus ((T_x M)^\perp \cap T_x \mathbb{H}^{p,q}) \oplus (T_x \mathbb{H}^{p,q})^\perp .
	\]
	We can now deduce the claimed characterization. Assume that $(x,u)$ is a critical point. By applying \eqref{eq:differential2} to $(\dot{x},\dot{u}) = (0, \dot{u})$, we see that $\langle \dot{u} , z^T \rangle = 0$ for every $\dot{u} \in u^\perp \cap T_x M$. This implies that $z^T$ belongs to $\mathrm{Span}(u)$ and, hence, $z^T = \langle u, z^T\rangle \, u$. Applying again \eqref{eq:differential2} to $(\dot{x},\dot{u}) = (\dot{x},0)$ and replacing $z^T = \langle u, z^T\rangle \, u$, we obtain
	\[
	\langle \dot{x} , (F_z(x,u)^2 - 1)\, \langle x, z \rangle \, u - (B_{z^N})_x(u) \rangle = 0
	\]
	for every $\dot{x} \in T_x M$, which implies $(B_{z^N})_x(u) = (F_z(x,u)^2 - 1) \, \langle x, z \rangle \, u$. The opposite implication follows directly from \eqref{eq:differential2}.
\end{proof}

We now have all the ingredients to prove the remaining technical statement.

\begin{proof}[Proof of Lemma~\ref{lem:maximum principle}]
	We will adopt the same notations used in the proof of Lemma~\ref{lem:critical point}. We start by determining the second order derivative of $F$ along the curve $t \mapsto (\gamma(t),U(t))$, where:
	\begin{itemize}
		\item $\gamma : (-\varepsilon,\varepsilon) \to M$ is a geodesic arc of $M$ that satisfies $\gamma(0) = x$ and $\gamma'(0) = \dot{x}$.
		\item $U \in \Gamma(\gamma^*(T^1 M))$ denotes the parallel transport of $u$ along $\gamma$. 
	\end{itemize}
	It follows that
	\begin{align}
		D^0_{\partial_t} \gamma' & = \mathbb{I}(\gamma',\gamma') + \langle \gamma', \gamma' \rangle \, \gamma , \label{eq:geodesic} \\
		D^0_{\partial_t} U & = \mathbb{I}(\gamma',U) + \langle \gamma', U \rangle \, \gamma . \label{eq:parallel vector field}
	\end{align}
	Moreover, being $\gamma'$ and $U$ both $\nabla$-parallel vector fields, they satisfy $\langle \gamma' , U \rangle = \langle \gamma' , U \rangle|_{t = 0} = \langle \dot{x}, u \rangle$. 
	
	Assume now that $(x,u) \in T^1 M$ is a point of maximum for $F$. In particular, $(x,u)$ is a critical point of $F$ and, hence, 
	\begin{gather}
		z^T = \langle u, z^T \rangle \, u, \label{eq:projection tangent} \\
		\qquad  (B_{z^N})_x(u) = \langle x, z \rangle \, (L^2 - 1) \, u,\label{eq:eigenvalue shape}
	\end{gather}
	by Lemma~\ref{lem:critical point}. (Recall that $L : = F_z(x,u)$.) It is not restrictive to assume that $L > 1$, in which case $\langle u , z \rangle \neq 0$. To simplify the notation, we set $\hat{z} : = \langle u,z^T \rangle^{-1} z$ and, accordingly, $\hat{z}^T = u$, $\hat{z}^N = \langle u,z^T \rangle^{-1} \, z^N$. By applying equation \eqref{eq:first order derivative} and recalling that $F(\gamma,U)'\,|_{t = 0} = 0$, we have
	\begin{align*}
		F(\gamma,U)'\,|_{t = 0} & = \left.\left( F_z(\gamma,U) \, \frac{\langle D^0_{\partial_t} U , z \rangle}{\langle U, z \rangle} - F_z(\gamma,U)^3 \frac{\langle \gamma , z \rangle \, \langle \gamma' , z \rangle}{\langle U,z \rangle^2} \right)'\, \right|_{t = 0} \\
		& = L \left. \left( \frac{\langle \mathbb{I}(\gamma',U), z \rangle + \langle \dot{x}, u \rangle \, \langle \gamma , z \rangle}{\langle U, z \rangle} \right)'\, \right|_{t = 0} - L^3 \left. \left( \frac{\langle \gamma , z \rangle \, \langle \gamma' , z \rangle}{\langle U,z \rangle^2} \right)' \, \right|_{t = 0} \\
		& = L (\langle D^0_{\partial_t}(\mathbb{I}(\gamma',U))|_{t = 0}, \hat{z} \rangle + \langle \dot{x}, u \rangle \, \langle \dot{x}, \hat{z}^T \rangle ) \\
		& \qquad - L \, (\langle \dot{x}, (B_{\hat{z}^N})_x(u) \rangle + \langle \dot{x}, u \rangle \, \langle x, \hat{z} \rangle ) \, \langle D^0_{\partial_t} U|_{t = 0}. \hat{z} \rangle  \\
		& \qquad - L^3(\langle \dot{x}, \hat{z}^T \rangle^2 + \langle x,\hat{z} \rangle \, \langle D^0_{\partial_t} \gamma'|_{t = 0} , \hat{z} \rangle ) \\
		& \qquad +2 L^3 \, \langle x, \hat{z} \rangle \, \langle \dot{x}, \hat{z}^T \rangle \, \langle D^0_{\partial_t} U|_{t = 0}. \hat{z} \rangle \\
		& = L \, \langle D^0_{\partial_t}(\mathbb{I}(\gamma',U))|_{t = 0}, \hat{z} \rangle + L\, ( 1 - L^2 + L^2 \, \langle x, \hat{z} \rangle^2) \, \langle \dot{x}, u \rangle^2  \\
		& \qquad - L^3 \, \langle x,\hat{z} \rangle \, (\langle \mathbb{I}_x(\dot{x},\dot{x} - \langle \dot{x}, u \rangle \, u), \hat{z}^N \rangle + \langle \dot{x}, \dot{x} \rangle \, \langle x , \hat{z} \rangle),
	\end{align*}
	where in the last line we simplified the expression by applying the identities \eqref{eq:geodesic}, \eqref{eq:parallel vector field}, $\hat{z}^T = u$, and $(B_{\hat{z}^N})_x(u) = \langle x, \hat{z}\rangle (L^2 -1) u$.
	
	Let us focus for a moment on the term $D^0_{\partial_t}(\mathbb{I}(\gamma',U))|_{t = 0}$. The following holds:
	\begin{align*}
		\langle D^0_{\partial_t}(\mathbb{I}(\gamma',U))|_{t = 0}, \hat{z} \rangle & = \langle D_{\partial_t}(\mathbb{I}(\gamma',U))|_{t = 0} + \langle \dot{x}, \mathbb{I}_x(\dot{x},u) \rangle \, x, \hat{z} \rangle \\
		& = \langle (D_{\dot{x}}\mathbb{I})_x(u,\dot{x}), \hat{z} \rangle \\
		& = \langle (D_u \mathbb{I})_x(\dot{x},\dot{x}) + R^\nabla(\dot{x},u)\dot{x} - \langle \dot{x} , u \rangle \, \dot{x} + \langle \dot{x} , \dot{x} \rangle \, u, \hat{z} \rangle \\
		& = \langle (D_u \mathbb{I})_x(\dot{x},\dot{x}), \hat{z}^N \rangle + R^\nabla(\dot{x}, u, \dot{x}, u) - \langle \dot{x}, u \rangle^2 + \langle \dot{x}, \dot{x} \rangle \\
		& = \langle (D_u \mathbb{I})_x(\dot{x},\dot{x}), \hat{z}^N \rangle + \langle \mathbb{I}(\dot{x},\dot{x}), \mathbb{I}(u,u) \rangle - \langle \mathbb{I}(\dot{x},u), \mathbb{I}(\dot{x},u) \rangle  ,
	\end{align*}
	where: In the first equality we expressed the covariant derivative $D_\bullet(\mathbb{I}(\gamma',U))$ as the orthogonal projection (with respect to $\langle \bullet, \bullet \rangle = \langle \bullet, \bullet \rangle_{p, q+1}$) of the (flat) derivative $D^0_\bullet(\mathbb{I}(\gamma',U))$ onto $T_\bullet \widehat{\mathbb{H}}^{p,q}$. In the second step we noticed that $\dot{x}$ is orthogonal to the image of $\mathbb{I}_x$. In the third and fourth equalities, we made use of the Gauss-Codazzi equations \eqref{eq:Gauss-Codazzi} and the expression of the Riemann tensor from \eqref{eq:riemann hpq}.
	
	By combining this identity with the expression of the second order derivative of $F$ from above, we deduce:
	\begin{align*}
	(\mathrm{Hess} \, F)_{(x,u)}&((\dot{x},0), (\dot{x},0)) = F(\gamma,U)''\,|_{t = 0} \\
	& = L \left( \langle (D_u \mathbb{I})_x(\dot{x},\dot{x}), \hat{z}^N \rangle + \langle \mathbb{I}(\dot{x},\dot{x}), \mathbb{I}(u,u) \rangle - \langle \mathbb{I}(\dot{x},u), \mathbb{I}(\dot{x},u) \rangle \right) \\
	& \qquad + L \, (1 - L^2 + L^2 \langle x, \hat{z} \rangle^2) \, \langle \dot{x}, u \rangle^2 \\
	& \qquad - L^3 \, \langle x,\hat{z} \rangle \, (\langle \mathbb{I}_x(\dot{x},\dot{x} - \langle \dot{x}, u \rangle \, u), \hat{z}^N \rangle + \langle \dot{x}, \dot{x} \rangle \, \langle x , \hat{z} \rangle) .
	\end{align*}
	(Notice that, being $(x,u)$ a critical point, the Hessian of $F$ is a well-defined bilinear form on $T_{(x,u)}(T^1 M)$ that does not depend on the choice of a connection on $T^1 M$.) Since $(x,u)$ is a local maximum, we must have $(\mathrm{Hess} \, F)_{(x,u)}((\dot{x},0), (\dot{x},0)) \leq 0$ for every $\dot{x} \in T_x M$. By taking the trace of the bilinear form 
	\[
	(\dot{x}, \dot{x}') \mapsto (\mathrm{Hess} \, F)_{(x,u)}((\dot{x},0), (\dot{x}',0))
	\]
	with respect to the metric $g_M$, we deduce:
	\begin{align}\label{eq:laplacian at max}
		\begin{split}
			0 & \geq L ( - \tr_M(\langle \mathbb{I}(\bullet,u), \mathbb{I}(\bullet,u) \rangle) + 1 - L^2 + L^2 \, \langle x, \hat{z} \rangle^2) \\
			& \qquad + L^3 \langle x, \hat{z} \rangle \, \langle (B_{\hat{z}^N})_x(u),u \rangle - p \, L^3 \, \langle x, \hat{z} \rangle^2 \\
			& = L( - \, \tr_M(\langle \mathbb{I}(\bullet,u), \mathbb{I}(\bullet,u) \rangle) + 1 - L^2 + L^2 \, \langle x, \hat{z} \rangle^2 \, (L^2 - p) ) ,
		\end{split}
	\end{align}
	where in the last step we used again that $(B_{\hat{z}^N})_x(u) = \langle x, \hat{z}\rangle (L^2 -1) u$.
	
	The last necessary ingredient is an estimate from below of $\tr_M(\langle \mathbb{I}(\bullet,u), \mathbb{I}(\bullet,u) \rangle)$. For this purpose, we select bases $(e_i)_{i = 1}^p$ and $(f_j)_{j = 1}^q$ of $T_x M$ and $(T_x M)^\perp \cap T_x \widehat{\mathbb{H}}^{p,q}$, respectively, so that $e_1 = u$ and $f_1 = (- \langle \hat{z}^N, \hat{z}^N \rangle)^{-1/2} \hat{z}^N$. (Notice that $\langle \hat{z}^N, \hat{z}^N \rangle = L^{-2} - 1$, in particular $\hat{z}^N \neq 0$.) Then we have:
	\begin{align}\label{eq:stima norma 2}
		\begin{split}
			- \tr_M(\langle \mathbb{I}(\bullet,u), \mathbb{I}(\bullet,u) \rangle) & = \sum_{j = 1}^{q} \sum_{i = 1}^p \langle \mathbb{I}(u,e_i), f_j \rangle^2 \\
			& \geq - \frac{\langle \mathbb{I}(u,u), \hat{z}^N \rangle^2}{\langle \hat{z}^N, \hat{z}^N \rangle} \\
			& = - \frac{\langle (B_{\hat{z}^N})_x(u), u  \rangle^2}{L^{-2} - 1} \\
			& = L^2 \, \langle x, \hat{z} \rangle^2 (L^2 - 1) .
		\end{split}
	\end{align}
	Notice that in the second line we kept the addend with $(i,j) = (1,1)$ and we estimated from below all the remaining terms with $0$. By applying this inequality in \eqref{eq:laplacian at max} and noticing that $L^2 \, \langle x, \hat{z} \rangle^2 = 1 - \langle \hat{z} , \hat{z} \rangle \, L^2$, we finally obtain
	\begin{align}\label{eq:last step}
		\begin{split}
			0 & \geq L( - \, \tr_M(\langle \mathbb{I}(\bullet,u), \mathbb{I}(\bullet,u) \rangle) + 1 - L^2 + L^2 \, \langle x, \hat{z} \rangle^2 \, (L^2 - p) ) \\
			& \geq L( L^2 \, \langle x, \hat{z} \rangle^2 (L^2 - 1 - p) - L^2 + 1 ) \\
			& = L(L^2 - p - \langle \hat{z}, \hat{z} \rangle \, L^2(2 L^2 - p -1)) .
		\end{split}
	\end{align}
	Since $L > 0$, this implies the desired statement.
\end{proof}

	\begin{rmk}
		Let us briefly discuss the case in which the maximal value $\max F = \sqrt{p}$ is achieved at some point of $M$. In what follows, we address point (2) of Theorem~\ref{thm:main3}.
				
		The argument from above shows that $\max F_z = \sqrt{p}$ only when $z \in \partial M \subset \partial \widehat{\mathbb{H}}^{p,q}$. In such case, $F_z(x,u) = \langle \nabla \beta|_x, u \rangle = \sqrt{p}$ for some $(x,u) \in T^1 M$, where 
		\[
		u = \hat{z}^T = \frac{1}{\sqrt{p}} \, \nabla \beta|_x .
		\]
		Notice also that $\langle x, \hat{z} \rangle = \frac{\langle x,z \rangle }{\langle u,z\rangle} = (\dd\beta)_x(u)^{-1} = \frac{1}{\sqrt{p}}$.
		
		Furthermore, $L = F_z(x,u) = \sqrt{p}$ implies that all inequalities in \eqref{eq:last step} must be equalities and, hence, by relation \eqref{eq:stima norma 2}, it follows that
		\[
		\mathbb{I}_x(u,u) = \frac{\langle \mathbb{I}_x(u,u), \hat{z}^N \rangle}{\langle \hat{z}^N, \hat{z}^N \rangle} \, \hat{z}^N = - \sqrt{p} \, \hat{z}^N , \qquad \mathbb{I}_x(u,v) = 0 ,
		\]
		for every $v \in T_x M \cap u^\perp$ and, in particular, $- \tr_M(\langle \mathbb{I}_x(\bullet, u), \mathbb{I}_x(\bullet, u) \rangle) = p - 1$. Being $M$ maximal, the contracted Gau{\ss} equation \eqref{eq:traccia gauss} implies in particular that
		\[
		\mathrm{Ric}_M|_x(u,u) = - (p - 1) - \tr_M(\langle \mathbb{I}_x(\bullet,u), \mathbb{I}_x(\bullet,u) \rangle) = 0 .
		\]
		
		This fact, combined with the Bochner's formula applied to the function $\beta = b_{z,o}$ at $x \in M$, tells us that
		\begin{align}\label{eq:bochner}
			\begin{split}
				0 \geq \frac{1}{2} \, \Delta \norm{\nabla \beta}^2|_x & = \norm{\nabla^2 \beta}_x^2 + \langle \nabla \beta |_x , \nabla \Delta \beta|_x \rangle + \mathrm{Ric}_M (\nabla \beta,\nabla \beta)|_x \\
				& = \norm{\nabla^2 \beta}_x^2 + \sqrt{p} \, \langle u, \nabla \Delta \beta |_x \rangle + p \, \mathrm{Ric}_M|_x (u,u)  \\
				& = \norm{\nabla^2 \beta}_x^2 + \sqrt{p} \, \langle u, \nabla \Delta \beta |_x \rangle ,
			\end{split}
		\end{align}
		where $\Delta \norm{\nabla \beta}^2|_x \leq 0$ since $x$ is a point of maximum for $\norm{\nabla \beta}^2$. By the second part of Proposition~\ref{prop:hessian of busemann}, the term $\nabla \Delta \beta|_x$ verifies
		\[
		\nabla \Delta \beta|_x = - \nabla (\norm{\nabla \beta}^2)|_x = 0 ,
		\]
		since $x$ is a critical point of $\norm{\nabla \beta}^2$. It follows from \eqref{eq:bochner} that $\norm{\nabla^2 \beta}_x^2 = 0$.
	\end{rmk}

\subsection{Proof of Theorem~\ref{thm:distance comparison}}

We are now ready to prove Theorem~\ref{thm:distance comparison}.

\begin{proof}[Proof of Theorem~\ref{thm:distance comparison}]
Let $\gamma:[0,1]\to M$ be a geodesic for the intrinsic metric of $ M$ joining $x=\gamma(0)$ and $y=\gamma(1)$. In particular $d_ M(x,y)=\ell(\gamma)$. We have
\begin{align*}
d_{\mb{H}^{p,q}}(x,y) &=d_{\mb{H}^{p,q}}(x,y)-d_{\mb{H}^{p,q}}(x,x)\\
 &=\int_0^1{\langle\nabla d_x|_{\gamma(t)},{\dot \gamma}(t)\rangle \dd{t}}.
\end{align*}

By Theorem~\ref{thm:main3}, we have $\abs{\langle \nabla d_x|_{\gamma(t)},{\dot \gamma}(t)\rangle}\le\norm{\nabla d_x}_{ M,\infty} \norm{{\dot \gamma}(t)}\le\sqrt{p} \, \norm{{\dot \gamma}(t)}$. Thus, we get
\[
d_{\mb{H}^{p,q}}(x,y)\le\sqrt{p}\, \ell(\gamma) = \sqrt{p}\, d_M(x,y).
\]

This, combined with Lemma~\ref{lem:distance}, concludes the proof of the theorem.
\end{proof}

\section{A fundamental domain}

Let $\rho:\Gamma\to{\rm SO}(p,q+1)$ be a convex cocompact representation with maximal open, convex domain of discontinuity $\Omega_\rho\subset\mb{H}^{p,q}$. As the arguments we describe in this section only depend on the image group $\rho(\Gamma)$, we will assume that $\rho$ is injective. 

The goal of the section is to construct a geometrically controlled fundamental domain $F\subset\Omega_\rho$ and prove finiteness of volume and diameter (Proposition~\ref{pro:finite vol diam}). We start by choosing arbitrarily a basepoint $o\in C_\rho$ where $C_\rho\subset\Omega_\rho$ is a $\rho(\Gamma)$-invariant proper convex subset on which $\rho(\Gamma)$ acts cocompactly. Define
\[
F(\gamma):=\{x\in\Omega_\rho \left|\;-\langle x,o\rangle\le-\langle x,\rho(\gamma) o\rangle\right.\}
\]
and set
\[
F:=\bigcap_{\gamma\in\Gamma}{F(\gamma)}.
\]

We prove the following properties.

\begin{lem}
\label{lem:fundamental domain}
The following holds:
\begin{enumerate}
\item{$F$ is convex and intersects every orbit of $\Gamma\curvearrowright\Omega_\rho$.}
\item{The closure $\overline{F}$ in $\mb{P}(\mb{R}^{p,q+1})$ is contained in $\mb{H}^{p,q}$.}
\end{enumerate}
\end{lem}

\begin{proof}
Since $C_\rho$ is convex and, hence, simply connected, there exists a lift $\widetilde{C}_\rho$ of $C_\rho$ to the $2$-fold cover $\{ v \in \mathbb{R}^{p,q+1} \mid \langle v,v \rangle = -1\} \to \mathbb{H}^{p,q}$. Furthermore, every $\theta \in \Lambda_\rho$ admits a representative $\tilde{\theta}$ whose positive projective class is asymptotic to $\widetilde{C}_\rho$, in which case $\langle \bullet, \tilde{\theta} \rangle < 0$ on $\widetilde{C}_\rho$. With abuse of notation, we will not distinguish between points of $C_\rho$ and $\widetilde{C}_\rho$, or between $\theta$ and a representative $\tilde{\theta}$ verifying the conditions above. 

We prove the two properties separately.

\smallskip

\noindent {\bf Property (1).} Convexity follows from the fact that $\Omega_\rho$ is itself convex and each $F(\gamma)$ is a linear half space (hence convex). 

\smallskip

\begin{claim}{A}
	For any $x \in \Omega_\rho$, the set $\{-\langle x,\rho(\gamma) o\rangle\}_{\gamma\in\Gamma}$ has a minimum $m$.
\end{claim}

\begin{proof}[Proof of Claim A]
	
Suppose by contradiction that we have a sequence $\gamma_n\in\Gamma$ of pairwise distinct elements such that 
\[
-\langle x,\rho(\gamma_n)o\rangle\to m: = \inf_{\gamma\in\Gamma}{\{-\langle x,\rho(\gamma)o\rangle\}} \in [-\infty, +\infty) .
\]
As $\{\gamma_n\}_{n\in\mb{N}}$ contains infinitely many distinct elements, up to subsequences we can assume that $\rho(\gamma_n)o\to\theta$ for some $\theta\in\Lambda_\rho$, that is, there exists a sequence of positive numbers $\lambda_n>0$ such that $\lambda_n\,\rho(\gamma_n)o \to\theta$ in $\mb{R}^{p,q+1}$. Observe that $\lambda_n^2=-\langle \lambda_n\gamma_n (o),\lambda_n\rho(\gamma_n)o\rangle\to-\langle\theta,\theta\rangle=0$ and
\[
-\langle x,\lambda_n\rho(\gamma_n)o\rangle\to-\langle x,\theta\rangle>0 ,
\]
since $x\in\Omega_\rho$. Then we get
\[
m= \lim_{n \to \infty}(-\langle x,\rho(\gamma_n)o\rangle) = \lim_{n \to \infty}(-\langle x,\lambda_n\rho(\gamma_n)o\rangle/\lambda_n) = +\infty ,
\]
which is absurd. 
\end{proof}

By the claim, if $\gamma\in\Gamma$ denotes an element realizing the minimum $-\langle x,\rho(\gamma)o\rangle$, then $\rho(\gamma)^{-1}x\in F$, which proves that $F$ intersects the orbit of $x$ for any $x \in \Omega_\rho$.

\smallskip

\noindent{\bf Property (2).} By \cite{DGK}, the closure $\overline{\Omega}_\rho$ in $\mb{P}(\mb{R}^{p,q+1})$ intersects the boundary at infinity $\partial\mb{H}^{p,q}$ in the limit set $\overline{\Omega}_\rho\cap\partial\mb{H}^{p,q}=\Lambda_\rho$. Therefore, in order to show that $\overline{F}\subset\mb{H}^{p,q}$ it is enough to prove:

\begin{claim}{B}
	$\overline{F}\cap\Lambda_\rho=\emptyset$. 
\end{claim}

\begin{proof}[Proof of Claim B]
We argue by contradiction. Suppose this is not the case and pick $\theta\in\overline{F}\cap\Lambda_\rho$. By convexity of $F$, the geodesic ray $[o,\theta]$ is entirely contained in $C_\rho\cap F$. Observe that it is also spacelike by the definition of $\Omega_\rho$.

As $\Gamma$ acts cocompactly on $C_\rho$, there is a compact subset $K\subset C_\rho$ and a sequence $\gamma_n\in\Gamma$ with $\rho(\gamma_n)o\to\theta$ such that $\rho(\gamma_n)(K)\cap[o,\theta]\neq\emptyset$. Pick $x_n\in\rho(\gamma_n)(K)\cap[o,\theta]$. By compactness of $K$, we can find some $B>0$ such that $-\langle x,y\rangle\le B$ for every $y \in K$. In particular $-\langle x_n,\rho(\gamma_n)o\rangle\le B$ for every $n$. By the definition of $F$, we also have $-\langle x_n,o\rangle\le-\langle x_n,\rho(\gamma_n)o\rangle\le B$. However, since $x_n\to\theta$ along the spacelike ray $[o,\theta]$, we have $-\langle x_n,o\rangle\to + \infty$ which contradicts the previous observation and concludes the proof of the claim.
\end{proof}

This concludes the proof of the lemma.
\end{proof}

We use the fundamental domain $F$ to show that the volume and diameter of $\Omega_\rho/\rho(\Gamma)$ are both finite.

\begin{cor}
\label{cor:finite invariants}
We have ${\rm vol}_\rho={\rm vol}(\Omega_\rho/\rho(\Gamma)),{\rm diam}_\rho={\rm diam}(\Omega_\rho/\rho(\Gamma))<\infty$.
\end{cor}

\begin{proof}
Let $F\subset\Omega_\rho$ be the fundamental domain provided by Lemma~\ref{lem:fundamental domain}. By Property (2), the closure $\overline{F}$ is a compact subset of $\mb{H}^{p,q}$. Hence ${\rm vol}(\Omega_\rho/\rho(\Gamma))\le{\rm vol}(\overline{F})<\infty$. Similarly, since $\overline{F}$ is compact and contained in $\mb{H}^{p,q}$, we have $\cosh{\rm diam}(\Omega_\rho/\rho(\Gamma))\le\max\{-\langle x,y\rangle|\;x,y\in\overline{F}\}<\infty$. 
\end{proof} 

\section{Volume, entropy, and topology}\label{sec:vol,entr,top}

In this section we prove Theorem~\ref{thm:main2}: We bound from below the product volume times entropy of a convex cocompact representation $\rho:\Gamma\to{\rm SO}(p,q+1)$ of the fundamental group of a closed $p$-manifold $\mathsf{M}$ by its simplicial volume, which is a topological invariant.

Let $\mathsf{M}_\rho$ denote the unique maximal spacelike $p$-submanifold of $\mathsf{N}_\rho = \Omega_\rho/\rho(\Gamma)$ provided by \cite{SST23}. We denote by $d_ M(\bullet,\bullet)$ the distance determined by the induced Riemannian metric of $M$, the universal cover of $\mathsf{M}_\rho$, and by
\[
h(\mathsf{M}_\rho):= \limsup_{R\to\infty}{\frac{\log\left|\{\gamma\in\Gamma \mid d_M(o,\gamma o)\le R\}\right|}{R}}
\]
the associated critical exponent. We need the following bound:

\begin{pro}
\label{pro:volume}
We have 
\[
\mu\cdot{\rm vol}(\mathsf{M}_\rho)\le{\rm vol}_\rho ,
\]
for some $\mu=\mu(p,q)>0$. 
\end{pro}

Let us briefly prove Theorem~\ref{thm:main2} assuming the proposition above.

\begin{proof}[Proof of Theorem~\ref{thm:main2}]
By Theorem~\ref{thm:distance comparison} and the definition of $\delta_\rho$, we get
\[
h(\mathsf{M}_\rho)\le\frac{1}{\sqrt{p}}\cdot\delta_\rho.
\]
(Compare with Proposition~\ref{pro:finite vol diam}.) Recall that the Riemannian metric $g$ on $\mathsf{M}_\rho$ satisfies ${\rm Ric}_ g\ge-(p-1) \, g$. Hence, by work of Gromov (see \cite{Gr}*{Section~2.5}), we have 
\[
h(\mathsf{M}_\rho)^p\ge\kappa \, \norm{\mathsf{M}}/{\rm vol}(\mathsf{M}_\rho) ,
\]
where $\kappa=\kappa(p)>0$ is a dimensional constant and $\norm{\mathsf{M}}$ is the simplicial volume of $\mathsf{M}$. It follows that 
\[
\delta_\rho^p\ge p^{p/2}\kappa\cdot\frac{\norm{\mathsf{M}}}{{\rm vol}(\mathsf{M}_\rho)}.
\]

By Proposition~\ref{pro:volume}, we conclude
\[
\delta_\rho^p\ge \frac{p^{p/2}\kappa}{\mu}\cdot\frac{\norm{\mathsf{M}}}{{\rm vol}_\rho}.
\] 

This concludes the proof of the theorem.
\end{proof}

It remains to prove the proposition.

\subsection{Proof of Proposition~\ref{pro:volume}}

Let
\begin{align*}
	N M & := \{ (x,w) \in T \mathbb{H}^{p,q} \mid x \in M, \, w \in (T_x M)^\perp \} , \\
	N^1 M & := \{ (x,n) \in N M \mid \langle n,n \rangle = -1 \} ,
\end{align*}
denote the normal bundle (resp. unit normal bundle) of $M\subset\mb{H}^{p,q}$, the unique $\rho$-invariant, maximal, spacelike $p$-submanifold of $\mb{H}^{p,q}$ and consider the restriction of the exponential map $\nu:N M\to\mb{H}^{p,q}$. We work with coordinates $(x,n;t)$, where $t \in (0,\pi/2)$ and $(x,n) \in N^1 M$. Using the connection $D$ of $\mathbb{H}^{p,q}$, we can describe the tangent space $T_{(x,n)}(N^1 M)$ as
\[
T_{(x,n)}(N^1 M) \cong \{ (\dot{x}, \dot{n}) \mid \dot{x} \in T_x M, \dot{n} \in \langle n \rangle^\perp \cap N_x M \}
\]

In such coordinates $(x,n;t) \in N^1 M \times (0,\infty)$, the pull-back metric $\nu^*g_{\mb{H}^{p,q}}$ can be written as
\begin{align*}
	\nu^* g_{\mathbb{H}^{p,q}}|_{(x,n;t)}&((\dot{x}, \dot{n},\dot{t}) , (\dot{x}, \dot{n},\dot{t})) = - \dot{t}^2 - \sin^2 t \cdot {g_{\mathbb{S}^{q-1}_x}}|_n(\dot{n}, \dot{n}) \\
	&\quad + g_ M|_x((\cos t \, \mathrm{id} + \sin t \, B_n)\dot{x},(\cos t \, \mathrm{id} + \sin t \, B_n)\dot{x}) ,
\end{align*}
where $g_{\mathbb{S}^{q-1}_x}$ denotes the standard (positive definite) metric on the unit sphere $\mathbb{S}^{q-1}_x : = N^1_x M$. It follows that the volume form $\nu^* g_{\mathbb{H}^{p,q}}$ can be expressed as
\begin{align*}
	\nu^*\mathrm{vol}_{\mathbb{H}^{p,q}}|_{(x,n;t)} & = (\sin t)^{(q-1)/2} \, \det(\cos t \, \mathrm{id} + \sin t \, B_n)^{1/2} \dd{t} \wedge \mathrm{vol}_{\mathbb{S}_x^{q-1}} \wedge \mathrm{vol}_M \\
	& =  \det(\cos t \, \mathrm{id} + \sin t \, B_n)^{1/2} \, \mathrm{vol}_{\mathbb{S}^q} \wedge \mathrm{vol}_{M} ,
\end{align*}
where in the second equality we observed that the form $(\sin t)^{(q-1)/2} \, \dd{t} \wedge \mathrm{vol}_{\mathbb{S}_x^{q-1}}$ coincides with the volume form of the unit $q$-sphere $\mathbb{S}_x^q \subset (T_x M)^\perp$. From this identity, we deduce that $\nu$ is a local diffeomorphism on $N^1 M \times (0,r)$, for any $r$ that satisfies
\[
\frac{1}{\tan r} \geq \norm{\mathbb{I}}_\infty : = \sup\{\norm{\mathbb{I}_x(\dot{x},\dot{x})} \mid  x \in M, \, \dot{x} \in T_x M, \norm{\dot{x}} = 1 \} .
\]

A priori, the image $\nu(N^1 M \times (0,r))$ may not be contained inside the open convex domain $\Omega_\rho$. To establish this, we observe:

\begin{lem}\label{lem:intorno nel dominio}
	Let $M$ be a complete, maximal, spacelike $p$-submanifold of $\mathbb{H}^{p,q}$. Then the image $\nu(N^1 M \times (0,r))$ is contained in the domain of discontinuity $\Omega_\rho$ for any $r \leq \arctan(1/\sqrt{p-1})$. 
\end{lem}

\begin{proof}
	As argued in the proof of Lemma~\ref{lem:fundamental domain}, the convex hull of the limit set $C_\rho$ admits a lift $\widetilde{C}_\rho$ to the $2$-fold cover $\{v \in \mathbb{R}^{p,q+1} \mid \langle v, v \rangle = -1 \}$. In particular, $M$ admits a lift inside $\widetilde{C}_\rho$, which we continue to denote with abuse by $M$. For any point $\theta$ in the limit set $\Lambda_\rho$, we select a representative $\tilde{\theta}$ such that $\langle \bullet, \tilde{\theta} \rangle < 0$ on $\widetilde{C}_\rho $. With these conventions, the maximal, open, convex domain of discontinuity $\Omega_\rho \supset C_\rho$ can be described as
	\[
	\Omega_\rho : = \mathbb{P}\left( \bigcap_{\theta \in \Lambda_\rho} \{ v \in \mathbb{R}^{p,q+1} \mid \langle v, \tilde{\theta} \rangle  < 0 \} \right)
	\]
	(See e.g. \cite{DGK}.) Being $\nu(N^1 M \times (0,r))$ connected and $M \subset \Omega_\rho$, it is enough to prove that, for any $u \in \nu(N^1 M \times (0,r))$ and for any $\theta \in \Lambda_\rho$, we have $\langle y , \tilde{\theta}\rangle \neq 0$. 
	
	To see this, let $u = \cos t \, y + \sin t \, n$, for some $y \in M \subset \widetilde{C}_\rho$, $n \in N_{x}^1 M$, and $t \in (0,r)$. Then we write
	\[
	\frac{\langle \tilde{\theta}, u \rangle}{\langle \tilde{\theta}, y \rangle} = \cos t + \sin t \, \frac{\langle \tilde{\theta}, n\rangle}{\langle \tilde{\theta}, y \rangle} .
	\]
	
	By Theorem~\ref{thm:restatement_thm:main3}, the orthogonal projection $\tilde{\theta}^N_y$ of $\tilde{\theta}$ onto $N_y M$ satisfies 
	\[
	0 \leq - \langle \tilde{\theta}^N_y, \tilde{\theta}^N_y \rangle \leq (p - 1) \, \langle \tilde{\theta}, y \rangle^2 .
	\]
	This bound, combined with the Cauchy-Schwarz inequality (recall that $N_y M$ is negative definite), implies
	\[
	\frac{\langle \tilde{\theta}, u \rangle}{\langle \tilde{\theta}, y \rangle} \geq \cos t - \sin t \, \sqrt{- \langle \tilde{\theta}^N_y, \tilde{\theta}^N_y \rangle} \sqrt{- \langle n, n \rangle} \geq \cos t - (p - 1)^{1/2} \, \sin t .
	\]
	Under our hypotheses, the right hand side is strictly positive, implying in particular that $\langle \tilde{\theta}, u \rangle \neq 0$. By letting $u$ vary in $\nu(N^1 M \times (0,r))$, we deduce the desired statement.
\end{proof}

It remains to establish sufficient conditions for $\nu$ to be injective on $N^1 M \times (0,r)$.

\begin{lem}
	\label{lem:embedding}
	Let $M$ be a complete, maximal, spacelike $p$-submanifold of $\mathbb{H}^{p,q}$ with second fundamental form $\mathbb{I}$. Then the map $\nu: N^1 M \times (0,r) \to \Omega_\rho$ is injective, where
	\[
	r = \arctan\left(\frac{1}{\max \{\sqrt{p - 1}, \norm{\mathbb{I}}_\infty\}}\right) .
	\]
\end{lem}

The strategy for the proof of this statement was suggested to us by Timoth\'e Lemistre, we are grateful to him for allowing us to use it here.

\begin{proof}[Proof of Lemma~\ref{lem:embedding}]
	As in Lemma~\ref{lem:intorno nel dominio}, we select a lift of $M$ to $\{v \in \mathbb{R}^{p,q+1} \mid \langle v, v \rangle = -1 \}$, which we continue to denote with abuse by $M$. Any $u \in \nu(N^1 M \times (0,r))$ can be expressed as $u = \cos t \, x + \sin t \, n$, for some $x \in M$, $n \in N^1_x M$, and $t \in (0,r)$. Then we define
	\[
	f_u(y) = - \langle u, y \rangle = -  \cos t \, \langle x , y \rangle - \sin t \, \langle n , y \rangle , \quad y \in M.
	\]
	
	Notice that $y \in M$ is a critical point of $f_u : M \to \mathbb{R}$ if and only if $u \in (T_y M)^\perp$. We clearly have that $x$ is a critical point for $f_u$, and $f_u(x) = \cos t > \cos r$. In fact, we will show:
	
	\begin{claim}{A}\label{claim:unique min}
	The point $x$ is the unique global minimum of $f_u$.
	%$f_u : M \to \mathbb{R}$ is proper and admits a unique point of local minimum $y$ inside $\{f_y > \cos r\}$, i.e. $y = x$.
	\end{claim}

Assuming Claim A, we can deduce that $\nu$ is injective on $N^1 M \times (0,r)$: Let $x, y \in M$, $n \in N_x M$, $m \in N_y M$, and $t, s \in (0,r)$ be such that
	\[
	u = \cos t \, x + \sin t \, n = \cos s \, y + \sin s \, m .
	\]
	Then both $x$ and $y$ would be points of global minima for $f_u$. By the uniqueness part of Claim~\ref{claim:unique min}, we must have $x = y$ and $\cos t = \cos s$ (so $t = s$, since $t, s < r \leq \frac{\pi}{4})$ and, consequently, $(x,n;t) = (y,m;s) \in N^1 M \times (0,r)$.
	
	In order to prove Claim A, we need the following properties:
	\begin{enumerate}
		\item $f_u(y)$ tends to $+ \infty$ as $y$ goes towards $\partial M \subset \partial \mathbb{H}^{p,q}$.
		\item Every critical point $y$ of $f_u$ that satisfies $f_u(y) > \cos(r)$ is a strict local minimum.
	\end{enumerate}
	
	\begin{proof}[Proof of Claim A]
	 Recall that $x$ is a point of strict local minimum. In particular there exists a small disk $D_x\subset M$ around $x$ with $f_u>f_u(x)$ on $\partial D_x$. Assume there exists another point $y\in M-D_x$ with $f_u(y)\le f_u(x)$. By the mountain pass theorem (see \cite{Evans}*{Section~8.5, Theorem~2}), which we can apply as $f_u:M\to\mb{R}$ is a proper (by Property (1)) smooth function on $M$ (which is diffeomorphic to $\mb{R}^p$), there exists a critical point $z\in M$ with 
	\[
	f_u(z)=\inf_{\gamma:[0,1]\to M,\;\gamma(0)=x,\gamma(1)=y}{\max\{f_u\gamma\}}\ge f_u(x)=\cos(t).
	\]
	
	Furthermore, by work of Pucci and Serrin \cite{PucciSerrin}*{Corollary~2} on the structure of the critical set relative to the critical value $f_u(z)$, we can assume that $z$ is either a saddle or a local maximum.
	
	However, since $f_u(z)\ge f_u(x)=\cos(t)>\cos(r)$, by Property (2), we must have that $z$ is a point of strict local minimum. This gives a contradiction and shows that $x$ is the unique global minimum of $f_u$.
	\end{proof}
	
	We now focus on the proofs of Properties (1) and (2). We proceed following their order. To determine a bound from below of $f_u$, we observe:
		\begin{align*}
			f_u(y) & = \cos t \cosh d_{\mathbb{H}^{p,q}}(x,y) + \sin t \, \langle x, n \rangle \\
			& \geq \cos t \cosh d_{\mathbb{H}^{p,q}}(x,y) - \sin t \, \abs{\langle y_x^N, n \rangle} ,
		\end{align*}
		where $y_x^N$ denotes the orthogonal projection of $y$ onto $N_x M$. Recalling the estimates from Theorem~\ref{thm:restatement_thm:main3} and combining it with the Cauchy-Schwarz inequality on $N_x M$ (which is negative definite), we obtain:
		\begin{align*}
			f_u(y) - \cos s \, \cosh d_{\mathbb{H}^{p,q}}(x,y) & \geq \cos t \cosh d_{\mathbb{H}^{p,q}}(x,y)  - \sqrt{p - 1} \sin t \sinh d_{\mathbb{H}^{p,q}}(x,y) \\
			& \geq \left(\cos t - \sqrt{p - 1} \sin t\right) \cosh d_{\mathbb{H}^{p,q}}(x,y) .
		\end{align*}
		The hypothesis $t < r \leq \arctan(1/\sqrt{p-1})$ guarantees that the left multiplicative factor is positive. We deduce in particular that $f_u(y) \to + \infty$ as $y$ tends towards $\partial M \subset \partial \mathbb{H}^{p,q}$.

		Let us start by noticing that $f_u$ is the restriction on $M$ of a linear functional over $\mathbb{R}^{p,q+1}$. An elementary computation shows that any linear functional on $\mathbb{R}^{p,q+1}$ restricts to a function $\lambda : \mathbb{H}^{p,q} \to \mathbb{R}$ that satisfies $D^2_{X,Y} \lambda = \lambda \, \langle X, Y \rangle$ for any $X, Y$ tangent vector fields on $\mathbb{H}^{p,q}$, where $D^2_{X,Y} \lambda$ denotes the Hessian of $\lambda$ with respect to the Levi-Civita connection $D$ of $\mathbb{H}^{p,q}$. If we further restrict $\lambda$ to the maximal $p$-submanifold $M$, then its Hessian with respect to the Levi-Civita connection $\nabla$ of $M$ verifies
		\begin{equation}\label{eq:hessian}
			\nabla_{U,V}^2 \lambda = \lambda \, \langle U, V \rangle + \dd{\lambda}(D_U V - \nabla_U V) = \lambda \, \langle U, V \rangle + \dd{\lambda}(\mathbb{I}(U,V)) ,
		\end{equation}
		for any $U,V$ tangent vector fields to $M$.
		
		Assume now that $y \in M$ is a critical point for $f_y$ that satisfies $f_u(u) > \cos r$. In this case, both vectors $u + \langle u , y \rangle y = u - f_u(y) \, y$ and $\mathbb{I}_y(v,w)$ lie inside $N_y M$ (which is negative definite) for any $v,w \in T_y M$. From the fact that $u - f_u(y) \, y$ is timelike, we deduce that $f_u(y)^2 \leq 1$. Moreover, by applying the Cauchy-Schwarz inequality and relation \eqref{eq:hessian}, we deduce:
		\begin{align*}
			\nabla^2_{v,v} f_u|_y & \geq f_u(y) \langle v, v \rangle - \abs{\langle u - f_u(y) y, \mathbb{I}(v,v) \rangle} \\
			& \geq (f_u(y) - \sqrt{1 - f_u(y)^2} \norm{\mathbb{I}}_\infty) \langle v, v \rangle
		\end{align*}
		The hypothesis $t < r$ guarantees that the multiplicative factor in the last line is strictly positive, implying that $y$ is a point of strict local minimum. This concludes the proof of Property (2) and, hence, of Lemma~\ref{lem:embedding}.
\end{proof}

Now we have all the elements to conclude the proof of Proposition~\ref{pro:volume}. As observed above, the volume form $\nu^* \mathrm{vol}_{\mathbb{H}^{p,q}}$ can be expressed as
\[
	\nu^*(\mathrm{vol}_{\mathbb{H}^{p,q}})|_{(x,n;t)} =  \det(\cos t \, \mathrm{id} + \sin t \, B_n)^{1/2} \, \mathrm{vol}_{\mathbb{S}^q} \wedge \mathrm{vol}_{M} ,
\]
for any $(x,n;t) \in N^1 M \times (0,\infty)$. We select now 
\[
t_0 : = \arctan\left( \frac{1}{2 \max\{\sqrt{p - 1}, \norm{\mathbb{I}}_\infty\}} \right) .
\]
By Lemmas~\ref{lem:intorno nel dominio} and \ref{lem:embedding}, the restriction of the exponential map on $N^1 M \times (0,t_0)$ is an embedding inside $\Omega_\rho$. Our goal is now to find a bound from below of the volume of the open subset $\nu(N^1 M \times (0,t_0))/\rho(\Gamma)$ inside $\mathsf{N}_\rho = \Omega_\rho/\rho(\Gamma)$.

By definition, $\norm{B_n(\dot{x})} \leq \norm{\mathbb{I}}_\infty \, \norm{\dot{x}}$. It follows that all eigenvalues of $B_n$ are bounded by $\norm{\mathbb{I}}_\infty$ in absolute value. If $(\lambda_i(n))_{i = 1}^p$ denote the eigenvalues of $B_n$ ordered in decreasing order, then for any $t < t_0$ we have:
\begin{align*}
	\det(\cos t \, \mathrm{id} + \sin t \, B_n)^{1/2} & = \prod_{i = 1}^p (\cos t + \sin t \, \lambda_i(n))^{1/2} \\
	& = (\cos t)^{p/2} \prod_{i = 1}^p (1 + \tan t \, \lambda_i(n))^{1/2} \\
	& \geq (\cos t)^{p/2} (1 - \tfrac{1}{2})^{p/2} \\
	& = 2^{-p/2} \, (\cos t)^{p/2} ,
\end{align*}
since $\lambda_i(n) > - \norm{\mathbb{I}}_\infty$ and $\tan t \leq (2 \norm{\mathbb{I}}_\infty)^{-1}$. By integrating the volume form over $\nu(N^1 M \times (0,t_0))/\rho(\Gamma)$ and applying the bound from above, we deduce
\[
\mathrm{vol}(\mathsf{N}_\rho) \geq c(p) \, \mathrm{vol}(B_{\mathrm{S}^q}(t_0)) \, \mathrm{vol}(\mathsf{M}_\rho) ,
\]
where $B_{\mathbb{S}^q}(t_0)$ denotes the volume of the metric ball of radius $t_0 < \pi/2$ inside the unit sphere $\mathbb{S}^q$, and $c(p)$ is some positive constant. By Ishihara's bound on $\norm{\mathbb{I}}_\infty$ (see \eqref{eq:ricci bound}), the volume $B_{\mathbb{S}^q}(t_0)$ is bounded below by a constant that depends only on $p$ and $q$. This provides the uniform lower bound on ${\rm vol}_\rho$ as desired and finishes the proof of Proposition~\ref{pro:volume}.

\section{Entropy rigidity}
In this section we establish the inequality $\delta_\rho\le p - 1$ for convex cocompact representations $\rho:\Gamma\to{\rm SO}(p,q+1)$ of $p$-manifold groups $\Gamma$ and we characterize the cases where the equality is achieved. This proves Theorem~\ref{thm:main1} from the introduction.

\begin{proof}[Proof of Theorem~\ref{thm:main1}]
By the work of Seppi, Smith, and Toulisse \cite{SST23} there exists a unique spacelike $\rho$-invariant $p$-manifold $M\subset\mb{H}^{p,q}$ which is {\em maximal}, that is it has vanishing mean curvature.

We make the following observations:
\begin{itemize}
\item{Recall that, by Property \eqref{eq:ricci bound} the Ricci curvature of $M$ satisfies
\[
{\rm Ric}_M+(p-1)g_M\ge 0.
\]
}
\item{As the action $\rho(\Gamma)\curvearrowright M$ is cocompact, the Riemannian critical exponent $h(\mathsf{M}_\rho)$ (see Section~\ref{sec:vol,entr,top}) coincides with the volume entropy of $\mathsf{M}_\rho = M/\rho(\Gamma)$:
\[
h(\mathsf{M}_\rho) = \limsup_{R\to\infty}{\frac{\log{\rm vol}_M(B(o,R))}{R}} ,
\]
where $B_M(o,R)$ is the ball of radius $R$ around a fixed basepoint $o\in M$ for the intrinsic Riemannian metric of $M \subset \mathbb{H}^{p,q}$.}
\item{By Bishop-Gromov and the last identity, it follows that $h(\mathsf{M}_\rho)$ is bounded from above by $p-1$.}
\end{itemize}

By the definitions of $\delta_\rho$ and $h(\mathsf{M}_\rho)$ and Lemma~\ref{lem:distance}, we have $\delta_\rho\le h(\mathsf{M}_\rho)$ and, combining with the previous facts, we conclude that 
\[
\delta_\rho\le p-1.
\] 

If equality holds, then 
\[
\delta_\rho=h(\mathsf{M}_\rho)=p-1.
\]

By a rigidity result due to Ledrappier and Wang \cite{LW09}, since ${\rm Ric}_M\ge-(p-1)g_M$ the identity $h(\mathsf{M}_\rho)=p-1$ holds only when $\mathsf{M}_\rho$ is hyperbolic. The conclusion of the rigidity part of the statement then boils down to the following:

\begin{claim}{A}
	If $\mathsf{M}_\rho$ is hyperbolic, then it is totally geodesic.
\end{claim}

\begin{proof}[Proof of Claim A]
This is again a consequence of the Gauss equation. Denote by $\norm{\bullet}_t^2=-\langle\bullet,\bullet\rangle$ for timelike vectors. Recall that
\[
{\rm Ric}_{\mathsf{M}_\rho}(u,u)=-(p-1)\norm{u}^2+\sum_{j\le q}{\norm{\mb{I}(u,v_j)}_t^2}
\]
for every $x\in \mathsf{M}_\rho,u\in T_x \mathsf{M}_\rho$, and $v_1,\cdots,v_p$ orthonormal basis of $T_x \mathsf{M}_\rho$. If $\mathsf{M}_\rho$ is hyperbolic, then ${\rm Ric}_{\mathsf{M}_\rho}(u,u)=-(p-1)\norm{u}^2$ and therefore $\norm{\mb{I}(u,v_j)}_t^2=0$ for every $j$ and $u\in T_x M$. This implies that the second fundamental form of $\mathsf{M}_\rho$ vanishes $\mb{I}=0$ and hence that $\mathsf{M}_\rho$ is totally geodesic.
\end{proof}
This concludes the proof of Theorem~\ref{thm:main1}.
\end{proof}

\section{Finiteness and compactness}

Let $\Gamma$ be the fundamental group of a closed aspherical $p$-manifold $\mathsf{M}$ and $\rho:\Gamma\to{\rm SO}(p,q+1)$ a faithful convex cocompact representation. We prove that:
\begin{itemize}
\item{If ${\rm diam}_\rho\le D$ then $\mathsf{M}$ is homeomorphic to a manifold in a finite list only depending on $p$ and $D$.}
\item{The diameter function $\rho\in\mc{CC}(\Gamma)\to{\rm diam}_\rho\in(0,\infty)$ is proper.}
\end{itemize}
 
The main tool for the proof are the finiteness and compactness results in Riemannian geometry and Theorem~\ref{thm:distance comparison} that provides a bilipschitz relation between the length spectra of $\rho$ and of the associated maximal $p$-submanifold $\mathsf{M}_\rho$.

\subsection{Proof of Theorem~\ref{thm:main4}}
Consider $M\subset\mb{H}^{p,q}$, the unique $\rho$-invariant maximal $p$-submanifold provided by \cite{SST23}. Recall that, by the Gauss equation \eqref{eq:Gauss-Codazzi}, the intrinsic distance on $ M$ satisfies ${\rm Ric}_ M\ge-(p-1)g_ M$. Also recall that the intrinsic metric on $M$ is smaller than the induced pseudo-metric, that is $d_ M(\bullet,\bullet)\le d_{\mb{H}^{p,q}}(\bullet,\bullet)$. In particular, the intrinsic diameter of $\mathsf{M}_\rho$ is bounded from above by the pseudo-diameter ${\rm diam}_\rho$ which, by assumption, is bounded by ${\rm diam}_\rho\le D$.  

We are almost in the setting which enables us to use the following finiteness result of Anderson:

\begin{thm}[Anderson \cite{A}]
\label{thm:anderson}
Fix $p\in\mb{N}$, $k \in \mathbb{R}$, and $v,D\in\mb{R}^+$. Then the set
\[
\mathfrak{M}(p,k,v,d):=\left\{\pi_1(\mathsf{M}')\left|
\begin{array}{c}
\text{$(\mathsf{M}',g)$ closed Riemannian $p$-manifold with}\\
{\rm Ric}_g\ge k \, g,{\rm vol}(\mathsf{M}',g)\ge v,{\rm diam}(\mathsf{M}',g)\le D\\
\end{array}
\right.\right\}
\]
contains finitely many groups up to isomorphism.
\end{thm}

In order to apply Theorem~\ref{thm:anderson} we need to check that the volume of $\mathsf{M}_\rho$ cannot be too small.

\begin{claim}{A}
	There exists $\ep_p>0$ only depending on $p$ such that ${\rm vol}(\mathsf{M}_\rho)\ge\ep_p$.
\end{claim}

\begin{proof}[Proof of Claim A]
We proceed as follows: Since $\rho$ is convex cocompact, we have that $\Gamma$ is Gromov hyperbolic (see \cite{DGK}). By work of Mineyev \cite{M} and the fact that $\Gamma=\pi_1(\mathsf{M}_\rho)$, this implies that the simplicial volume of $\mathsf{M}_\rho$ is positive (and independent of $\rho$). The conclusion of the claim follows from the Isolation Theorem of Gromov (see \cite{Gr}), which states that there exists a dimensional constant $\ep_p>0$ such that if $(\mathsf{M}',g)$ is a closed Riemannian $p$-manifold with ${\rm Ric}_g\ge-(p-1)\,g$ and ${\rm vol}(\mathsf{M}',g)<\ep_p$ then the simplicial volume of $\mathsf{M}'$ vanishes. As in our setting $\norm{\mathsf{M}_\rho}>0$ and ${\rm Ric}_{\mathsf{M}_\rho}\ge-(p-1)\,g_{\mathsf{M}_\rho}$, we must have ${\rm vol}(\mathsf{M}_\rho)\ge\ep_p$.
\end{proof}

So far, we proved that, up to group isomorphism, there are only finitely many fundamental groups of closed aspherical $p$-manifolds that admit faithful $\mb{H}^{p,q}$-convex cocompact representations. By the solution of the Borel conjecture for hyperbolic groups by Bartels and L\"{u}ck \cite{BL}, one can upgrade finiteness up to isomorphism of fundamental groups to finiteness up to homeomorphism as claimed in Theorem~\ref{thm:main4}. 

\begin{rmk}
	Notice that, in the argument above, it is enough to know that $\mathrm{diam}(\mathsf{M}_\rho)$ is bounded from above.
\end{rmk}

\subsection{The proof of Theorem~\ref{thm:main5}}
Let $\rho_n:\Gamma\to{\rm SO}(p,q+1)$ be a sequence of representations such that ${\rm diam}(\rho_n)\le D$. 

Let $ M_n\subset\mb{H}^{p,q}$ be the unique $\rho_n(\Gamma)$-invariant maximal $p$-manifold provided by \cite{SST23}, and set $\mathsf{M}_{n} : = M_n / \rho_n(\Gamma)$. By the Gauss equation we have
\[
\abs{{\rm sec}_{ M_n}}\le\abs{{\rm sec}_{\mb{H}^{p,q}}}+\norm{\mb{I}_{ M_n}}^2\le 1+p^2q^2.
\]
where the last inequality follows from Ishihara's bound (see \cite{I}, or \eqref{eq:ricci bound}). 

Recall from the proof of the first part that ${\rm vol}( \mathsf{M}_{n})\ge\ep_p$ and ${\rm diam}( \mathsf{M}_{n})\le D$. Thus, we are in the setup of the Cheeger-Gromov Compactness Theorem (see \cite{P}). In particular, there exists $n_0>0$ such that, for every $n\ge n_0$, we can find a $2$-Lipschitz diffeomorphism $f_n: \mathsf{M}_{n_0}\to \mathsf{M}_{n}$. We consider the representations $ \rho_n':=\rho_n\circ(f_n)_*$, where $(f_n)_*:\Gamma\to\Gamma$ is the map induced by $f_n$ at the level of the fundamental groups. 

We show that $\rho_n'$ admits a convergent subsequence. In order to do so it is enough to show the following: Let $S$ be a finite set of generators for $\Gamma$.

\begin{claim}{B}
	Up to subsequences $\rho_n'(\gamma)$ converges for every $\gamma\in S$.
\end{claim}

\begin{proof}[Proof of Claim B]
We consider the action of $\rho_n'(\Gamma)$ on the symmetric space $\mathbb{X}_{p,q+1}$ of ${\rm SO}(p,q+1)$. 

By \cite{DGK}, for each $\gamma\in\Gamma$ and $n\ge n_0$, the element $\rho_n'(\gamma)$ has simple eigenvalue of largest modulus $\lambda_{{\rm max}}(\rho_n'(\gamma))$ greater than $1$, being $\rho_n'$ convex cocompact. Moreover, the minimal displacement on $\mathbb{X}_{p,q+1}$ of an isometry $A\in{\rm SO}(p,q+1)$ with simple largest eigenvalue $\lambda_{{\rm max}}( A)>1$ is
\[
\delta(A):=\min\{d_{\mathbb{X}_{p,q+1}}(x, A x)\mid x\in\mathbb{X}_{p,q+1}\} ,
\]
which is bounded by
\[
\delta(A)\le(p+q+1)\log\lambda_{{\rm max}}(A).
\]

For a finite subset $S\subset\Gamma$, we can consider the {\em joint minimal displacement} 
\[
\delta(\rho_n'(S))=\min_{x\in\mathbb{X}_{p,q+1}}\max_{\gamma\in S}\{d_{\mathbb{X}_{p,q+1}}(x,\rho_n'(\gamma)x)\}.
\]

Note that, if $\delta(\rho_n'(\gamma))$ is uniformly bounded independently of $n$, then the fact that ${\rm SO}(p,q+1)$ acts properly on $\mathbb{X}_{p,q+1}$ implies that $\rho_n'(\gamma)$ admits a convergent subsequence. Therefore, it is enough to show that $\delta(\rho_n'(S))$ is uniformly bounded independently of $n$.

By work of Breuillard and Fujiwara \cite{BF}*{Proposition~1.6}, we have
\begin{align*}
\delta(\rho_n'(S)) &\le L \, \max_{\gamma\in S^k}\{\delta(\rho_n'(\gamma))\}+ L \\
 &\le L \, \max_{\gamma\in S^k}\{(p+q+1)\log\lambda_{{\rm max}}(\rho_n'(\gamma))\}+ L ,
\end{align*}
for some dimensional constants $L=L(p,q)>0$ and $k=k(p,q)\in\mb{N}$ where $S^k$ is the set of elements in $\Gamma$ that can be written as a product of at most $k$ elements in $S$. As a consequence, given $k>0$, we need to find a uniform upper bound for $\log(\lambda_{{\rm max}}(\rho_n'(\gamma)))$, as $\gamma$ varies in $S^k$.

Consider one such element $\gamma\in S^k-\{1\}$. Let $\ell_n\subset\Omega_{\rho_n'}$ be the unique $\rho_n'(\gamma)$-invariant spacelike geodesic and pick a point on it $x_n\in\ell_n$. Let $u_n\in T_{x_n}\mb{H}^{p,q}$ be a timelike vector orthogonal to $\ell_n$ pointing towards $M_n$, that is, such that $y_n:=\cos(t_n)x_n+\sin(t_n)u_n\in M_n$ for some $t_n>0$. 

By Theorem~\ref{thm:distance comparison}, we have
\begin{align*}
\cosh(\sqrt{p}\cdot d_{ M_n}(y_n, \rho_n'(\gamma)^jy_n)) &\ge\cosh(d_{\mb{H}^{p,q}}(y_n,\rho_n'(\gamma)^jy_n))\\
 &=-\langle y_n,\rho_n'(\gamma)^jy_n\rangle.
\end{align*}
for any $j \in \mathbb{N}$. In turn,
\begin{align*}
-\langle y_n,\rho_n'(\gamma)^jy_n\rangle &=-\cos(t_n)^2\langle x_n,\rho_n'(\gamma)^jx_n\rangle-\sin(t_n)^2\langle u_n,\rho_n'(\gamma)^ju_n\rangle\\
 &=\cos(t_n)^2\cosh(j\log\lambda_{{\rm max}}(\rho_n'(\gamma)))-\sin(t_n)^2\langle u_n,\rho_n'(\gamma)^ju_n\rangle.
\end{align*}

As $\langle u_n,\rho_n'(\gamma)^ju_n\rangle$ is getting infinitesimal compared to $\cosh(j\log\lambda_{{\rm max}}(\rho_n'(\gamma)))$ (being $\lambda_{{\rm max}}(\rho_n'(\gamma))$ the largest eigenvalue of $\rho_n'(\gamma)$), putting together the previous inequalities we get
\[
\log\lambda_{{\rm max}}(\rho_n'(\gamma))\le\sqrt{p}\cdot\liminf_{j\to\infty}{\frac{d_{ M_n}(y_n,\rho_n'(\gamma)^jy_n)}{j}}.
\]

Observe that the right hand side coincides with the length of the shortest geodesic in the free homotopy class $[f_n(\gamma)]$ of $\mathsf{M}_{n} = M_n/\rho_n(\Gamma)$. As the map $f_n: \mathsf{M}_{n_0} \to \mathsf{M}_{n}$ is $2$-Lipschitz, we conclude that 
\[
\ell_{\mathsf{M}_n}(f_n(\gamma))\le 2\, \ell_{\mathsf{M}_{n_0}}(\gamma)
\]
which shows that $\lambda_{{\rm max}}(\rho_n'(\gamma))$ is uniformly bounded independently of $n$.
\end{proof}

\begin{rmk}
	Notice that, in the argument above, we only used that $\mathrm{diam}(\mathsf{M}_{\rho_n})$ is bounded uniformly in $n$, rather than $\mathrm{diam}_\rho \geq \mathrm{diam}(\mathsf{M}_{\rho_n})$. In particular, the analog statement of Theorem~\ref{thm:main5}, obtained by replacing the role of $\mathrm{diam}_\rho$ with $\mathrm{diam}(\mathsf{M}_\rho)$, also holds true.
\end{rmk}

\subsection{The proof of Theorem~\ref{thm:main6}}

Let $\Sigma$ be a closed orientable surface of genus at least $2$, let $\pi = \pi_1(\Sigma)$ denote its fundamental group, and select $\ep,V>0$ and $q\geq 1$. 

By Theorem~\ref{thm:main5}, the mapping class group ${\rm Out}(\pi)$ acts cocompactly on 
\[
\left\{\rho\in{\rm Hom}(\pi,{\rm SO}(2,q+1))\left| \, \,
	\rho\,\,{\rm maximal}, {\rm diam}_{\rho}\le D
\right.\right\}.
\]

Therefore it is enough to show that ${\rm vol}_{\rho}\le V$ and ${\rm sys}_{\rho}\ge\ep$ imply that ${\rm diam}_{\rho}\le D$ for some $D = D(V,\ep) > 0$. 

Since the unique $\rho(\pi)$-invariant maximal surface $M$ of $\mathbb{H}^{2,q}$ is non-positively curved (see Labourie and Toulisse \cite{LT22}), the following properties hold:
\begin{enumerate}
	\item The injectivity radius of the maximal surface $\mathsf{M}_\rho = M / \rho(\pi)$ coincides with half of its systole. The latter, by \cite{LT22}*{Proposition~5.8} (or Theorem~\ref{thm:main3}), is uniformly comparable with the systole $\mathrm{sys}_\rho$.
	\item The volume of any non-positively curved Riemannian manifold $\mathsf{M}'$ is bounded from below by $D \, \mathrm{vol}(B(r_{\mathsf{M}'}))$, where $B(r_{\mathsf{M}'})$ denotes the Euclidean ball of radius $r_{\mathsf{M}'} > 0$ in $\mathbb{R}^{\dim \mathsf{M}'}$ and $r_{\mathsf{M}'}$ is some explicit function of $\mathrm{sys}_{\mathsf{M}'}$.
\end{enumerate}

By the properties above and Proposition~\ref{pro:volume}, the diameter of $\mathsf{M}_\rho$ satisfies
\[
\mathrm{diam}(\mathsf{M}_\rho) \leq \frac{\mathrm{vol}(\mathsf{M}_\rho)}{\mathrm{vol}(B(r_{\mathsf{M}_\rho}))} \leq \frac{\mathrm{vol}_{\rho}}{\mu \, \mathrm{vol}(B(r_{\mathsf{M}_\rho}))} ,
\]
which implies the desired statement, as $r_{\mathsf{M}_\rho}$ is bounded from below by some (explicit) function of $\ep > 0$ by item (1).

\end{document}